\documentclass[final,leqno]{siamltex704}
\usepackage{amsmath}
\usepackage{graphicx}
\usepackage[notcite,notref]{showkeys}
\usepackage{mathrsfs}
\usepackage{graphics}
\usepackage{float}
\usepackage{amsfonts,amssymb}
\usepackage{dsfont}
\usepackage{pifont}
\usepackage{wrapfig} % Allows in-line images
\usepackage{hyperref}
\usepackage{multirow}
\usepackage{color}
\numberwithin{equation}{section}
\usepackage{threeparttable}
\def\3bar{{|\hspace{-.02in}|\hspace{-.02in}|}}
\def\E{{\mathcal{E}}}
\def\T{{\mathcal{T}}}

\def\pT{{\partial T}}

\def\W{{\mathcal{W}}}

\def\bn{{\mathbf{n}}}

\def\ljump{{[\![}}
\def\rjump{{]\!]}}

\def\bbeta{{\boldsymbol{\beta}}}

\newtheorem{example}{\bf Example}%[section]

\newtheorem{algorithm}{Primal-Dual Weak Galerkin Algorithm}[section]

\title {An \textbf{$L^p$}- Primal-Dual Weak Galerkin Method for Convection-Diffusion Equations}

\begin{document}

\author{
Waixiang Cao \thanks{School of Mathematical Sciences, Beijing Normal University, Beijing 100875, China (caowx@bnu.edu.cn). The research of Waixiang Cao was partially supported by NSFC grant No. 11871106.}
\and
Chunmei Wang \thanks{Department of Mathematics, University of Florida, Gainesville, FL 32611, USA (chunmei.wang@ufl.edu). The research of Chunmei Wang was partially supported by National Science Foundation Award DMS-1849483.}
\and
Junping Wang\thanks{Division of Mathematical
Sciences, National Science Foundation, Alexandria, VA 22314
(jwang@nsf.gov). The research of Junping Wang was supported in part by the
NSF IR/D program, while working at National Science Foundation.
However, any opinion, finding, and conclusions or recommendations
expressed in this material are those of the author and do not
necessarily reflect the views of the National Science Foundation.}}

\maketitle

\begin{abstract}
In this article, the authors present a new $L^p$- primal-dual weak Galerkin method ($L^p$-PDWG) for convection-diffusion equations with $p>1$. The existence and uniqueness of the numerical solution is discussed, and an optimal-order error estimate is derived in the $L^q$-norm for the primal variable, where $\frac 1p+\frac 1q=1$. Furthermore, error estimates are established for the numerical approximation of the dual variable in the standard $W^{m,p}$ norm, $0\le m\le 2$. Numerical results are presented to demonstrate the efficiency and accuracy of the proposed $L^p$-PDWG method.
\end{abstract}



\pagestyle{myheadings}

\section{Introduction}
In this paper, the authors are concerned with the development of an $L^p$- primal-dual weak Galerkin ($L^p$-PDWG) finite element method for second order elliptic boundary value problems that seek $u$ such that
\begin{equation}\label{model}
\begin{split}
- \Delta u + \nabla\cdot(\bbeta u) &=f, \quad \text{in} \ \Omega,\\
u&=g, \quad \text{on} \  \partial \Omega,
\end{split}
\end{equation}
where $\Omega\subset \mathbb R^d$($d = 2, 3$) is an open bounded and connected domain with piecewise smooth Lipschitz boundary $\partial \Omega$. By introducing the space $V=W_0^{1,p}(\Omega) \cap W^{2,p}(\Omega)$, we obtain the following weak form for the model problem (\ref{model}) following an use of the usual integration by parts: Find $u\in L^q(\Omega)$ such that
\begin{equation}\label{weakform}
 (u,  \Delta \sigma + \bbeta\cdot\nabla \sigma)= \langle g, \nabla \sigma\cdot\bn \rangle_{\partial \Omega} - (f, \sigma), \qquad \forall \sigma\in V.
\end{equation}
Assume the model problem \eqref{model} has one and only one solution for any given $f$ and $g$ in appropriate spaces.

The weak Galerkin  (WG) method  was first introduced by Wang and Ye in \cite{wang-ye-2014} for second order elliptic equations, where weak gradient and its discrete weak gradient were constructed to replace the standard gradient and its discrete gradient.  Later, the authors in \cite{WW-mathcomp} developed a primal-dual weak Galerkin (PDWG) finite element method for the second order elliptic problem in non-divergence form, where a {\it discrete weak Hessian} operator  in the weak formulation of the model PDEs was designed. This PDWG algorithm can be characterized as a constrained $L^2$ optimization problem with constraints given by the weak formulation of the model PDEs weakly defined on each element. In the past several years, many theoretical a priori error estimates for weak Galerkin methods have been established in $L^2$ and discrete $H^m, \ m=1,2$ norms.  Readers are referred to \cite{ww2017, ww2018,  wwhyperbolic,  wang-ye-2014,wang-ye-2015, w1, w2, w3, w4, w5, w6, w7} for an incomplete list of references. 

The purpose of this paper is to present an $L^p$- primal-dual weak Galerkin method for the problem \eqref{model}, and establish a general $L^q$ and $L^p$ theory for the numerical method. To our best knowledge, there is one existing result in the $L^p\ (1<p< \infty)$ error estimate for the mixed finite element method developed by Duran \cite{r1988} for second order elliptic problems in $\mathbb R^2$, but no results in $L^p$ are known for the weak Galerkin finite element methods in the literature. Different from the method in \cite{WW-mathcomp}, our numerical scheme is based on the weak formulation \eqref{weakform} together with a weak version of the dual operator applied to the test functions.  The new PDWG method can be characterized as a constrained $L^p$ optimization problem with constraints that satisfy the PDE weakly on each element, which extends the idea of $L^2$ minimization problem in \cite{WW-mathcomp} to a more general $L^p$ setting.

To conclude this section, we point out that our theory for the $L^p$ primal-dual weak Gelerkin finite element method is based on the assumption that the solution to the following adjoint problem
\begin{equation}\label{EQ:adjoint}
\begin{split}
-\Delta \varphi - \bbeta\cdot\nabla \varphi = &\ \chi,\qquad \mbox{ in } \Omega,\\
\varphi = & \ 0, \qquad \mbox{ on } \partial\Omega,
\end{split}
\end{equation}
is $W^{2,p}$-regular in the sense that it has a unique solution in $W^{1,p}_0(\Omega)\cap W^{2,p}(\Omega)$ and the solution satisfies
$$
\|\varphi\|_{2,p} \leq C \|\chi\|_{0,p}.
$$
Under this assumption, we shall derive an optimal order error estimate in the standard $L^q$ norm for the primal variable $u$ and the standard $W^{m,p}, 0\le m\le 1$ norms for the dual variable. Numerical experiments demonstrate that our error estimate for the primal variable is optimal; i.e., the error bound is sharp.

The rest of this paper is organized as follows. In Section \ref{Section:Hessian}, we briefly review the weak differential operators and their discrete versions. In Section \ref{Section:WGFEM}, the primal-dual weak Galerkin scheme is introduced for the model problem \eqref{model} based on $L^p$ theory. Section \ref{Section:EU} is devoted to the establishment of the solution existence, uniqueness and stability. In Section \ref{Section:EE}, we derive an  error equations for our numerical methods, which is  of essential importance in our later error estimates.
 Section \ref{Section:ES-1} and Section \ref{Section:ES-2} establish  error estimates for the primal variable in $L^q$ norms and for the dual variable in $W^{m,p}, 0\le m\le 1$ norms, respectively. Finally, a series of numerical examples are presented in Section \ref{Section:NR} to verify the mathematical convergence theory.

\section{Weak Differential Operators}\label{Section:Hessian}
The Laplacian and the gradient are the principle differential operators used in the weak formulation (\ref{weakform}) for the second order elliptic model problem (\ref{model}). This section gives a brief discussion of the weak Laplacian and gradient operators as well as their discrete analogies \cite{w2020}.

Let $T$ be a polygonal or polyhedral domain with boundary $\partial T$. A weak function on $T$ refers to a triplet $\sigma=\{\sigma_0,\sigma_b, \sigma_n\}$ such that $\sigma_0\in L^p(T)$, $\sigma_b\in L^{p}(\partial T)$, and $\sigma_n\in L^{p}(\partial T)$. The first and second components $\sigma_0$ and $\sigma_b$ can be identified as the value of $\sigma$ in the interior and on the boundary of $T$. The third component $\sigma_n$ is meant to represent the value of $\nabla \sigma \cdot \bn$ on the boundary of the element $T$. Note that $\sigma_b$ and $\sigma_n$ might be totally independent of the trace of $\sigma_0$ and $\nabla \sigma_0  \cdot \bn$ on $\partial T$, respectively. Denote by $\W(T)$ the space of all scalar-valued weak functions on $T$; i.e.,
\begin{equation}\label{2.1}
\W(T)=\{\sigma=\{\sigma_0,\sigma_b, \sigma_n\}: \sigma_0\in L^p(T), \sigma_b\in L^{p}(\partial T), \sigma_n\in L^{p}(\partial T)\}.
\end{equation}
The weak Laplacian operator, denoted by $\Delta_w$, is defined as a linear functional in $W^{2,q}(T)$ such that
\begin{equation*}
(\Delta_w \sigma, w)_T= (\sigma_0, \Delta w)_T-\langle \sigma_b,  \nabla w\cdot \textbf{n}\rangle_{\partial T}+ \langle \sigma_n, w\rangle_{\partial T},
\end{equation*}
for all $w \in  W^{2,q}(T)$.

Denote by $P_r(T)$ the space of all polynomials on $T$ with degree no more than $r$. A discrete analogy of $\Delta_w \sigma$ for $\sigma\in \W(T)$ is defined as the unique polynomial $\Delta_{w, r, T} \sigma \in P_r(T)$ satisfying
\begin{equation}\label{Loperator1-1}
(\Delta_{w, r, T} \sigma, w)_T= (\sigma_0, \Delta w)_T-\langle \sigma_b,  \nabla w\cdot \textbf{n}\rangle_{\partial T}+\langle \sigma_n, w\rangle_{\partial T}, \quad\forall w \in P_r(T).
\end{equation}
For smooth $\sigma_0$ such that $\sigma_0\in W^{2,p}(T)$, we have from the integration by parts
\begin{equation}\label{Loperator1-2}
(\Delta_{w, r, T} \sigma, w)_T= (\Delta \sigma_0, w)_T+\langle \sigma_0-\sigma_b,  \nabla w\cdot
\textbf{n}\rangle_{\partial T}- \langle \nabla \sigma_0\cdot \bn-\sigma_n, w \rangle_{\partial T},
\end{equation}
for all $w \in P_r(T)$. Similarly, the discrete weak gradient operator is defined as the unique polynomial $\nabla_{w, r, T} \sigma \in [P_r(T)]^d$ satisfying
\begin{equation}\label{Gradient-1}
(\nabla_{w, r, T} \sigma, \varphi)_T= -(\sigma_0, \nabla\cdot \varphi)_T+\langle \sigma_b,  \varphi\cdot \bn\rangle_{\partial T}, \quad\forall \varphi \in [P_r(T)]^d.
\end{equation}
When $\sigma_0\in W^{1,p}(T)$, the following identify holds true:
\begin{equation}\label{Gradient-2}
(\nabla_{w, r, T} \sigma, \varphi)_T= (\nabla \sigma_0, \varphi)_T+\langle \sigma_b-\sigma_0,  \varphi\cdot
\bn\rangle_{\partial T},
\end{equation}
for all $\varphi \in [P_r(T)]^d$.

\section{Numerical Algorithm}\label{Section:WGFEM}
Denote by ${\cal T}_h$ a partition of the domain $\Omega$ into
polygons in 2D or polyhedra in 3D which is shape regular in the sense described in \cite{wang-ye-2014}. Denote by ${\mathcal E}_h$ the
set of all edges or flat faces in ${\cal T}_h$ and  ${\mathcal
E}_h^0={\mathcal E}_h \setminus
\partial\Omega$ the set of all interior edges or flat faces.
Denote by $h_T$ the meshsize of $T\in {\cal T}_h$ and
$h=\max_{T\in {\cal T}_h}h_T$ the meshsize for the partition
${\cal T}_h$.

For any given integer $k\geq 1$, denote by
$W_k(T)$ the local discrete space of  the weak functions defined by
$$
W_k(T)=\{\{\sigma_0,\sigma_b, \sigma_n\}:\sigma_0\in P_k(T),\sigma_b\in
P_k(e), \sigma_n \in P_{k-1}(e),e\subset \partial T\}.
$$
Patching $W_k(T)$ over all the elements $T\in {\cal T}_h$
through a common value of $\sigma_b$ and $\sigma_n$ on the interior interface $\E_h^0$ yields a weak finite element space $W_h$:
\begin{equation}\label{wh}
W_h=\big\{\{\sigma_0, \sigma_b, \sigma_n\}:\{\sigma_0, \sigma_b,  \sigma_n\}|_T\in W_k(T), \forall T\in {\cal T}_h \big\}.
\end{equation}
Note that $\sigma_n$ has two values $\sigma_n^L$ and $\sigma_n^R$ on each interior interface $e=\partial T_L\cap \partial T_R \in {\cal E}_h^0$ as seen from the two elements $T_L$ and $T_R$, and they must satisfy $\sigma_n^L+\sigma_n^R=0$. Denote by $W_h^0$ the subspace of $W_h$ with homogeneous boundary condition; i.e.,
\begin{equation*}\label{wh0}
 W_h^0=\{\{\sigma_0, \sigma_b, \sigma_n\}:\{\sigma_0, \sigma_b,  \sigma_n\}|_T\in W_h,\sigma_b|_{\partial\Omega}=0,\ \forall e\in\partial T, T\in{\mathcal T}_h \}.
\end{equation*}
Denote by $M_h$ the finite element space consisting of piecewise polynomials of degree $s$ where $s=k-1$; i.e.,
\begin{equation}\label{mh}
M_h=\{w: w|_T\in P_{s}(T),  \forall T\in {\cal T}_h\}.
\end{equation}
{\em We emphasize that both the weak gradient and the weak Laplacian operators are defined by using piecewise polynomials of degree $s=k-1$.} For purely diffusive equations, one may assume the value of $s=k-2$.

For simplicity of notation and without confusion, for any $\sigma\in
W_h$, denote by $\Delta_{w} \sigma$ and $\nabla_w\sigma$ the discrete weak Laplacian $\Delta_{w, s, T} \sigma$  and discrete weak gradient $\nabla_{w, s, T} \sigma$  computed by (\ref{Loperator1-1}) and \eqref{Gradient-1} on each element $T$, respectively; i.e.,
 $$
(\Delta_{w} \sigma)|_T=\Delta_{w, s, T}(\sigma|_T),\quad (\nabla_{w} \sigma)|_T=\nabla_{w, s, T}(\sigma|_T),\ s=k-1.
$$

For any $\sigma, \lambda\in W_h$ and $u\in M_h$, we introduce the following forms
\begin{align} \label{EQ:local-stabilizer}
s(\lambda, \sigma)=&\sum_{T\in {\cal T}_h}s_T(\lambda, \sigma),
\\
b(u, \lambda)=&\sum_{T\in {\cal T}_h}b_T(u, \lambda),  \label{EQ:local-bterm}
\end{align}
where
\begin{equation*}
\begin{split}
s_T(\lambda, \sigma)=&h_T^{1-2p}\int_{\partial T} |\lambda_0-\lambda_b|^{p-1}sgn(\lambda_0-\lambda_b)(\sigma_0-\sigma_b)ds\\
&+  h_T^{1-p}\int_{\partial T}  |\nabla \lambda_0 \cdot \bn-\lambda_n|^{p-1}sgn (\nabla \lambda_0 \cdot \bn-\lambda_n) (\nabla \sigma_0 \cdot \bn-\sigma_n)ds\\
%& + \int_{T} |\lambda_0|^{p-1}sgn(\lambda_0) \ \sigma_0 dT + \int_{T} |\nabla\lambda_0|^{p-1}sgn(\nabla\lambda_0) \nabla\sigma_0 dT,\\
b_T(u, \lambda)=&(u, - \bbeta\cdot\nabla_w\lambda - \Delta_w \lambda)_T.
\end{split}
\end{equation*}

The numerical scheme for the second order elliptic model problem (\ref{model}) based on the variational formulation (\ref{weakform}) can be stated as follows:
\begin{algorithm}
Find $(u_h;\lambda_h)\in M_h \times W_{h}^0$, such that
\begin{eqnarray}\label{32}
s(\lambda_h, \sigma)+b(u_h, \sigma)&=& (f, \sigma_0)- \langle g, \sigma_n \rangle_{\partial \Omega}, \qquad \forall \sigma\in W_{h}^0,\\
b(v, \lambda_h)&=&0,\qquad \qquad \qquad   \qquad \qquad \forall v\in M_h.\label{2}
\end{eqnarray}
\end{algorithm}

In next section, we shall study the solution existence and uniqueness for the primal-dual weak Galerkin finite element algorithm (\ref{32})-(\ref{2}). For simplicity of analysis, we assume constant value for the convection term $\bbeta$
 %and the reaction coefficient $c$
 on each element $T\in \T_h$ in the rest of the paper.

\section{Solution Existence and Uniqueness}\label{Section:EU}
Denote by $Q_0$ the $L^2$ projection operator onto $P_k(T)$ for each element $T$. For each edge or face $e\subset\partial T$, denote by $Q_b$ and $Q_n$ the $L^2$ projection operators onto $P_{k}(e)$ and $P_{k-1}(e)$, respectively. For any $w\in W^{2,p}(\Omega)$, define
by $Q_h w$ the $L^2$  projection onto the weak finite element
space $W_h$ such that on each element $T$,
$$
Q_hw=\{Q_0w,Q_bw, Q_n(\nabla w \cdot \bn)\}.
$$
Denote by ${\cal Q}_h$ the $L^2$ projection onto the finite element space $M_h$.

\begin{lemma}\label{Lemma5.1} \cite{w2020} The $L^2$ projection operators $Q_h$ and ${\cal Q}_h$ satisfy the following commutative properties:
\begin{eqnarray}\label{l}
\Delta_{w}(Q_h w) &=& {\cal Q}_h (\Delta w), \qquad  w\in W^{2,p}(T),\\
\nabla_{w}(Q_h v) &=& {\cal Q}_h (\nabla v), \qquad  v\in W^{1,p}(T).\label{l:2}
\end{eqnarray}
\end{lemma}

To show the existence of solutions, we consider the functional
\[
     J(\sigma,v) : = \frac{1}{p}s(\sigma,\sigma) + b(v, \sigma) - (F,\sigma),
 \]
where $(F,\sigma)=(f, \sigma_0)- \langle g, \sigma_n  \rangle_{\partial \Omega}$ and $p>1$. If $(u_h;\lambda_h)\in M_h \times W_{h}^0$ is the solution of \eqref{32}-\eqref{2}, then we have from \eqref{2}
  \[
      J(\lambda_h,v)=\frac{1}{p} s(\lambda_h,\lambda_h)  - (F,\lambda_h)= J(\lambda_h,u_h),\ \ \forall v\in M_h.
  \]
On the other hand, the equation \eqref{32} indicates that $\partial_{\sigma} J(\lambda_h, u_h)(\sigma)=0$ for all $\sigma\in W_h^0$, where $\partial_{\sigma} J(\lambda_h, u_h)(\sigma)$ is the Gateaux partial derivative at $\lambda_h$ in the direction of $\sigma$. It follows that $\lambda_h$ is a  global  minimizer of the functional $\sigma \rightarrow J(\sigma, u_h)$; i.e.,
$$
J(\lambda_h, u_h)\leq J(\sigma,u_h),\qquad \forall \sigma\in W^0_h.
$$
Consequently,
\begin{equation}\label{EQ:saddle-point}
J(\lambda_h,v)\le J(\lambda_h,u_h)\le J(\sigma,u_h), \qquad \forall v\in M_h, \sigma\in W^0_h.
\end{equation}
The above inequality implies  that the solution $(u_h; \lambda_h)$ is a saddle point of the functional $J(\cdot,\cdot)$. Thus, \eqref{32}-\eqref{2} can be formulated as the following min-max problem: Find $u_h\in M_h$ and $\lambda_h
\in W_{h}^0$ such that
$$
( \lambda_h, u_h) = \arg\min_{\sigma\in W^0_h}\max_{v\in M_h} J(\sigma,v).
$$
As a convex minimization problem, the above problem has a solution so that there must be a solution $(u_h;\lambda_h)$ satisfying  \eqref{32}-\eqref{2}.

The rest of this section is devoted to a discussion of the uniqueness of the numerical solution $(u_h; \lambda_h)$.

\begin{theorem}
The numerical scheme (\ref{32})-(\ref{2}) has one and only one solution $(u_h;\lambda_h)$ in the finite element space $M_h \times W_{h}^0$.
\end{theorem}
\begin{proof}
Let $(u^{(1)}_h;\lambda^{(1)}_h)$ and $(u^{(2)}_h;\lambda^{(2)}_h)$ be two solutions of (\ref{32})-(\ref{2}).  Denote
\[
    \epsilon_h=\lambda^{(1)}_h-\lambda^{(2)}_h=\{\epsilon_0,\epsilon_b,\epsilon_n\},\ \ e_h=u^{(1)}_h-u^{(2)}_h.
\]
    For any constants $\theta_1,\theta_2$,
  we choose $\sigma=\theta_1\lambda^{(1)}_h+\theta_2\lambda^{(2)}_h$ in \eqref{32} and use \eqref{2} to obtain
\[
     s(\lambda^{(1)}_h, \theta_1\lambda^{(1)}_h+\theta_2\lambda^{(2)}_h)-s(\lambda^{(2)}_h, \theta_1\lambda^{(1)}_h+\theta_2\lambda^{(2)}_h)=0.
\]
In particular,  by taking $(\theta_1,\theta_2)=(1,0), (0,1)$, we have
  \begin{equation}\label{33}
     s(\lambda^{(1)}_h, \lambda^{(1)}_h)=s(\lambda^{(2)}_h,\lambda^{(1)}_h),\ \  s(\lambda^{(2)}_h, \lambda^{(2)}_h)=s(\lambda^{(1)}_h,\lambda^{(2)}_h),
\end{equation}
   which yields, together with Young's inequality  $|AB|\le \frac{|A|^p}{p}+\frac{|B|^q}{q}$, that  \[
   s(\lambda^{(1)}_h, \lambda^{(1)}_h)\le \frac{s(\lambda^{(2)}_h,\lambda^{(2)}_h)}{q}+\frac{s(\lambda^{(1)}_h,\lambda^{(1)}_h)}{p},\ \
    s(\lambda^{(2)}_h, \lambda^{(2)}_h)\le \frac{ s(\lambda^{(1)}_h,\lambda^{(1)}_h)}{q}+\frac{ s(\lambda^{(2)}_h,\lambda^{(2)}_h)}{p},
\]
which yields
\begin{equation}\label{35}
    s(\lambda^{(1)}_h, \lambda^{(1)}_h)=s(\lambda^{(2)}_h, \lambda^{(2)}_h).
\end{equation}
On the other hand,  for any two real numbers $A,B$,  there holds
  \[
   \left|\frac{A+B}{2}\right|^p\le (|A|^p+|B|^p)/2,
 \]
and the equality holds true if and only if $A=B$.  It follows that
 \begin{equation}\label{34}
     s(\frac{\lambda^{(1)}_h+\lambda^{(2)}_h}{2}, \frac{\lambda^{(1)}_h+\lambda^{(2)}_h}{2}) \le  \frac{1}{2} \big( s( \lambda^{(1)}_h, \lambda^{(1)}_h)+ s( \lambda^{(2)}_h, \lambda^{(2)}_h)\big)=s( \lambda^{(1)}_h, \lambda^{(1)}_h).
 \end{equation}
      By \eqref{33}-\eqref{35} and Young's inequality,
\begin{eqnarray*}
    s( \lambda^{(1)}_h, \lambda^{(1)}_h)
     &= &\frac{1}{2}\big ( s( \lambda^{(1)}_h, \lambda^{(1)}_h)+ s( \lambda^{(1)}_h, \lambda^{(2)}_h)\big)
=s( \lambda^{(1)}_h, \frac{ \lambda^{(1)}_h+\lambda^{(2)}_h}{2}) \\
  & \le & \frac 1q{s(\lambda^{(1)}_h,\lambda^{(1)}_h)} +\frac 1p {s(\frac{\lambda^{(1)}_h+\lambda^{(2)}_h}{2}, \frac{\lambda^{(1)}_h+\lambda^{(2)}_h}{2})},
\end{eqnarray*}
  which indicates that
 \[
     s( \lambda^{(1)}_h, \lambda^{(1)}_h)\le {s(\frac{\lambda^{(1)}_h+\lambda^{(2)}_h}{2}, \frac{\lambda^{(1)}_h+\lambda^{(2)}_h}{2})}.
 \]
   In light of \eqref{34}, we easily obtain that
 \[
    s(\frac{\lambda^{(1)}_h+\lambda^{(2)}_h}{2}, \frac{\lambda^{(1)}_h+\lambda^{(2)}_h}{2}) =s( \lambda^{(1)}_h, \lambda^{(1)}_h)=s( \lambda^{(2)}_h, \lambda^{(2)}_h).
 \]
    The above equality holds true if and only if
   \begin{eqnarray*}
%\lambda_0^{(1)} &=& \lambda_0^{(2)},\mbox{ in } T, \\
       \lambda_0^{(1)}-\lambda_b^{(1)} &=& \lambda_0^{(2)}-\lambda_b^{(2)}, \mbox{ on } \pT,\\
\nabla \lambda^{(1)}_0\cdot{\bf n}-\lambda_n^{(1)}&=&\nabla \lambda^{(2)}_0\cdot{\bf n}-\lambda_n^{(2)}, \mbox{ on } \pT,
 \end{eqnarray*}
    or equivalently,
\begin{eqnarray}\label{e_0}
%\epsilon_0 &=& 0, \mbox{ in } T,\\
       \epsilon_0 &=& \epsilon_b, \mbox{ on } \pT,\\
\nabla \epsilon_0\cdot{\bf n}&=&\epsilon_n, \mbox{ on } \pT.\label{e_1}
 \end{eqnarray}
Let $(u^{(1)}_h;\lambda^{(1)}_h)$ and $(u^{(2)}_h;\lambda^{(2)}_h)$ be two solutions of (\ref{32})-(\ref{2}). We have from \eqref{2} that
 $b(v, \epsilon_h)=0$. Using \eqref{Loperator1-2} and \eqref{Gradient-2}, we have
 \begin{equation*}
 \begin{split}
 0=&b(v, \epsilon_h)\\
 =& \sum_{T\in {\cal T}_h} (v, -\bbeta\cdot \nabla_w \epsilon_h-\Delta_w \epsilon_h)_T\\
  =& \sum_{T\in {\cal T}_h} -(\nabla\epsilon_0, \bbeta v)_T-\langle \epsilon_b-\epsilon_0, \bbeta v\cdot\bn\rangle_{\partial T}\\
  &-(\Delta \epsilon_0, v)-\langle \epsilon_0-\epsilon_b, \nabla v \cdot\bn\rangle_{\partial T}+\langle \nabla\epsilon_0\cdot \bn-\epsilon_n, v\rangle_{\partial T}\\
  =& \sum_{T\in {\cal T}_h}  (-\bbeta\cdot\nabla\epsilon_0-\Delta \epsilon_0, v)_T,
 \end{split}
 \end{equation*}
 where we used  \eqref{e_0} -\eqref{e_1}, which gives
 $-\bbeta\cdot\nabla\epsilon_0-\Delta \epsilon_0=0$ on each $T\in {\cal T}_h$. Together with \eqref{e_0} -\eqref{e_1}, we arrive at $-\bbeta\cdot\nabla\epsilon_0-\Delta \epsilon_0=0$  in $\Omega$, with the boundary condition $\epsilon_0=0$ on $\partial \Omega$ due to the fact that \eqref{e_0}  and $\epsilon_h\in W_h^0$. Therefore, we obtain $\epsilon_0=0$ in $\Omega$. Furthermore,
  we have $\epsilon_b=0$ and $\epsilon_n=0$,
 which leads to $\lambda^{(1)}_h=\lambda^{(2)}_h$.

We next show $e_h=0$.  To this end, using  $\lambda^{(1)}_h=\lambda^{(2)}_h $ and the equation \eqref{32} we obtain
$$
b(e_h,\sigma)=s(\lambda^{(1)}_h,\sigma)- s(\lambda^{(2)}_h,\sigma) +b(e_h,\sigma)=0, \quad \forall\sigma\in W_h^0,
$$
which, together with the definition of the weak Laplacian $\Delta_w$ and the weak gradient $\nabla_w$, yields
\begin{eqnarray*}
  0=  b(e_h,\sigma)&=& \sum_{T\in\T_h}(e_h,  - \bbeta\cdot\nabla_w\sigma - \Delta_w\sigma)_T\\
& = & \sum_{T\in\T_h} ( \nabla\cdot(\bbeta e_h) - \Delta e_h, \sigma_0)_T \\
& & + \sum_{T\in\T_h} \langle\sigma_b,\nabla e_h\cdot\bn\rangle_\pT-\langle\sigma_n,e_h\rangle_\pT
-\langle e_h, \bbeta\cdot\bn \sigma_b\rangle_\pT\\
 &=&\sum_{T\in\T_h} (\nabla\cdot(\bbeta e_h) - \Delta e_h, \sigma_0)_T \\
&&+\sum_{e\in{\cal E}_h}\int_{e} \left(\ljump\nabla e_h-\bbeta e_h\rjump\cdot\bn_e \sigma_b- \ljump e_h\rjump\sigma_n\right)ds
\end{eqnarray*}
for all $\sigma\in W_h^0$, where $\bn_e$ is the assigned outward normal direction to $e\in {\cal E}_h$ and $\ljump\cdot\rjump$ is the jump on the edge $e\in {\cal E}_h$.  In particular, by taking $\sigma_0=\nabla\cdot(\bbeta e_h) - \Delta e_h$, $\sigma_n|_{{\cal E}_h}=-\ljump e_h\rjump$, and $\sigma_b|_{{\cal E}_h^0}=\ljump\nabla e_h-\bbeta e_h\rjump \cdot\bn_e$  we obtain on each $T\in\T_h$
\begin{eqnarray*}
   &&-\triangle e_h + \nabla\cdot(\bbeta e_h) =0,  \mbox{ in } T,\\
&&  \ljump e_h\rjump =0,\ \ \ljump\nabla e_h-\bbeta e_h\rjump\cdot\bn_e=0,  \mbox{ on } \pT.
\end{eqnarray*}
  Consequently, from the solution uniqueness for \eqref{model} we have
\[
   e_h\equiv 0,\ \mbox{ or equivalently } \ u_h^{(1)}=u_h^{(2)}.
\]
 This completes the proof.
\end{proof}

\section{Error Equation}\label{Section:EE}
Let $u$ and $(u_h;\lambda_h) \in M_h\times W_h^0$ be the exact solution of \eqref{model} and its numerical solution arising from the PDWG scheme (\ref{32})-(\ref{2}), respectively. Note that the exact solution of the Lagrangian multiplier $\lambda$ is $0$. Denote two error functions by
\begin{align}\label{error}
e_h&= {\cal Q}_hu-u_h,
\\
\epsilon_h&=  Q_h\lambda-\lambda_h=-\lambda_h.\label{error-2}
\end{align}

\begin{lemma}\label{errorequa}
Let $u$ and $(u_h;\lambda_h) \in M_h\times W_h^0$ be the exact solution of \eqref{model} and its numerical solution arising from  PDWG scheme (\ref{32})-(\ref{2}). The error functions $e_h$ and $\epsilon_h$ satisfy the following equations:
\begin{eqnarray}\label{sehv}
s(\epsilon_h, \sigma)+b(e_h, \sigma)&=& l_u(\sigma), \qquad \forall \sigma\in W_{h}^0,\\
b(v, \epsilon_h)&=&0,\qquad\qquad \forall  v\in M_h. \label{sehv2}
\end{eqnarray}
\end{lemma}
Here
 \begin{equation}\label{lu}
\begin{split}
\qquad l_u(\sigma)=&\sum_{T\in {\cal T}_h} \langle u-{\cal Q}_h u, \sigma_n-\nabla \sigma_0\cdot\bn\rangle_{\partial T}+\langle \nabla  u\cdot\bn-\nabla {\cal Q}_h  u\cdot\bn, \sigma_0-\sigma_b\rangle_{\partial T}\\
&+ \sum_{T\in {\cal T}_h}\langle (u-{\cal Q}_hu) \bbeta\cdot\bn, \sigma_b-\sigma_0\rangle_\pT.
\end{split}
\end{equation}

\begin{proof}  First, from (\ref{error-2}) and (\ref{2}) we may readily derive (\ref{sehv2}). Next, by using \eqref{EQ:local-bterm} for $b(\cdot,\cdot)$ and choosing $w={\cal Q}_hu$ in (\ref{Loperator1-2}) and \eqref{Gradient-2}, we have
\begin{equation}\label{err}
\begin{split}
&b({\cal Q}_hu, \sigma)\\
=&\sum_{T\in {\cal T}_h}({\cal Q}_hu,  - \bbeta\cdot\nabla_w\sigma-\Delta_w\sigma)_T\\
=& \sum_{T\in {\cal T}_h} ({\cal Q}_h u, -\bbeta\cdot\nabla\sigma_0-\Delta \sigma_0)_T
-\langle {\cal Q}_hu \bbeta\cdot\bn, \sigma_b-\sigma_0\rangle_\pT \\
&+\langle \nabla {\cal Q}_hu\cdot\bn, \sigma_b-\sigma_0\rangle_{\partial T}+\langle {\cal Q}_hu, \nabla \sigma_0\cdot\bn - \sigma_n\rangle_{\partial T}\\
=& \sum_{T\in {\cal T}_h} (u, -\bbeta\cdot\nabla\sigma_0-\Delta \sigma_0)_T
-\langle {\cal Q}_hu \bbeta\cdot\bn, \sigma_b-\sigma_0\rangle_\pT \\
&+\langle \nabla {\cal Q}_hu\cdot\bn, \sigma_b-\sigma_0\rangle_{\partial T}+\langle {\cal Q}_hu,\nabla \sigma_0\cdot\bn - \sigma_n\rangle_{\partial T}.
\end{split}
\end{equation}
Now applying the usual integration by parts to the integrals on $T$ yields
\begin{equation}\label{err:02}
\begin{split}
&b({\cal Q}_hu, \sigma)\\
=& \sum_{T\in {\cal T}_h} (-\Delta u+\nabla\cdot(\bbeta u),  \sigma_0)_T + \langle (u-{\cal Q}_hu) \bbeta\cdot\bn, \sigma_b-\sigma_0\rangle_\pT \\
&+\langle \nabla ({\cal Q}_hu - u)\cdot\bn, \sigma_b-\sigma_0\rangle_{\partial T}+\langle ({\cal Q}_hu-u),(\nabla \sigma_0\cdot\bn - \sigma_n\rangle_{\partial T} - \langle g, \sigma_n\rangle_{\partial \Omega}\\
=& \ (f,  \sigma_0)_T - \langle g, \sigma_n\rangle_{\partial \Omega}\\
& + \sum_{T\in {\cal T}_h}\langle \nabla ({\cal Q}_hu - u)\cdot\bn, \sigma_b-\sigma_0\rangle_{\partial T}+\langle ({\cal Q}_hu-u),(\nabla \sigma_0\cdot\bn - \sigma_n\rangle_{\partial T}\\
& + \sum_{T\in {\cal T}_h}\langle (u-{\cal Q}_hu) \bbeta\cdot\bn, \sigma_b-\sigma_0\rangle_\pT,
\end{split}
\end{equation}
where we have used $\eqref{model}$ and $\sigma_b=0$ on $\partial\Omega$.
From $\lambda=0$, we have  $s(Q_h\lambda, \sigma)=0$. Subtracting (\ref{32}) from (\ref{err:02}) yields the error equation (\ref{sehv}). This completes the proof of the lemma.
\end{proof}

The equations (\ref{sehv})-(\ref{sehv2}) are called {\em error
equations} for the primal-dual WG finite element scheme
(\ref{32})-(\ref{2}).

\section{$L^q$-Error Estimate for the Primal Variable $u_h$}\label{Section:ES-1}
Recall that $\T_h$ is a shape-regular finite element partition of
the domain $\Omega$. For any $T\in\T_h$ and $\nabla w\in L^{q}(T)$ with $q>1$, the following trace inequality holds true:
\begin{equation}\label{trace-inequality}
\|w\|^q_{L^q(\pT)}\leq C
h_T^{-1}(\|w\|_{L^{q}(T)}^q+h_T^{q} \| \nabla w\|_{L^{q}(T)}^q).
\end{equation}

\begin{theorem} \label{theoestimate-1} Let $q>1$ and $k\geq 1$. Let $u$ and $(u_h;\lambda_h) \in M_h\times W_h^0$ be the exact solution of the second order elliptic model problem \eqref{model} and the numerical solution arising from PDWG scheme (\ref{32})-(\ref{2}). The following error estimate holds true:
\begin{equation}\label{sestimate}
s(\lambda_h,\lambda_h) \leq Ch^{qk}\|\nabla^{k}u\|^q_{L^q(\Omega)}.
\end{equation}
 \end{theorem}

\begin{proof} Recall that $\epsilon_h = - \lambda_h$. By letting $\sigma = -\lambda_h$  in (\ref{sehv}),  we have from (\ref{sehv2}) and \eqref{lu} that
\begin{equation}\label{EQ:April7:001}
\begin{split}
s(\lambda_h, \lambda_h)
= &\sum_{T\in {\cal T}_h} \langle u-{\cal Q}_h u, \nabla \lambda_0\cdot\bn-\lambda_n\rangle_{\partial T}+\langle \nabla  u\cdot\bn - \nabla {\cal Q}_h  u\cdot\bn , \lambda_b-\lambda_0\rangle_{\partial T}\\
& + \sum_{T\in {\cal T}_h}\langle (u-{\cal Q}_hu) \bbeta\cdot\bn, \lambda_0-\lambda_b\rangle_\pT.
\end{split}
\end{equation}
For the first term on the right-hand side of \eqref{EQ:April7:001}, we use the Cauchy-Schwarz inequality to obtain
\begin{equation}\label{e1}
\begin{split}
&\Big|\sum_{T\in {\cal T}_h}\langle {\cal Q}_hu - u, \lambda_n-\nabla \lambda_0\cdot\bn\rangle_{\partial T}\Big|\\
 \leq & \Big(\sum_{T\in {\cal T}_h}\| u-{\cal Q}_hu\|^q_{L^q(\pT)}\Big)^{\frac{1}{q}} \Big(\sum_{T\in {\cal T}_h}\|  \lambda_n-\nabla \lambda_0\cdot\bn\|^p_{L^p(\pT)}\Big)^{\frac{1}{p}}.
\end{split}
\end{equation}
For the term $\| u-{\cal Q}_hu\|^q_{L^q(\pT)}$, we have from the trace inequality \eqref{trace-inequality} that
\begin{equation}\label{te2}
\begin{split}
\sum_{T\in {\cal T}_h}\| u-{\cal Q}_hu\|^q_{L^q(\pT)}
\leq &  \sum_{T\in {\cal T}_h}h_T^{-1}\Big(\| u-{\cal Q}_hu\|^q_{L^q(T)}+h_T^q\|\nabla(u-{\cal Q}_hu)\|^q_{L^q(T)}\Big)\\
% \leq  &  \sum_{T\in {\cal T}_h}h_T^{q-1} \|\nabla(u-{\cal Q}_hu)\|^q_{L^q(T)}\\
 \leq & Ch^{kq-1}\|\nabla^{k}u\|^q_{L^q(\Omega)}.
\end{split}
\end{equation}
Substituting (\ref{te2}) into (\ref{e1}) gives
\begin{equation}\label{e3}
\begin{split}
&\Big|\sum_{T\in {\cal T}_h}\langle {\cal Q}_hu - u, \lambda_n-\nabla \lambda_0\cdot\bn\rangle_{\partial T}\Big|\\
\leq & Ch^{k}\|\nabla^{k}u\|_{L^q(\Omega)} \Big(\sum_{T\in {\cal T}_h}h_T^{1-p} \|\lambda_n-\nabla \lambda_0\cdot\bn\|^p_{L^p(\pT)}\Big)^{\frac{1}{p}}\\
\leq & C_1  \sum_{T\in {\cal T}_h}  h_T^{1-p} \|\epsilon_n-\nabla \epsilon_0\cdot\bn\|^p_{L^p(\pT)} + C_2 h^{qk}\|\nabla^{k}u\|^q_{L^q(\Omega)}.
\end{split}
\end{equation}

As to the second term on the right-hand side in \eqref{EQ:April7:001}, using the Cauchy-Schwarz inequality we have
\begin{equation}\label{e2}
\begin{split}
&\Big|\langle \nabla {\cal Q}_h  u\cdot\bn - \nabla  u\cdot\bn  , \lambda_b-\lambda_0\rangle_{\partial T}\big| \\
\leq &  \Big(\sum_{T\in {\cal T}_h}\| \nabla  u\cdot\bn-\nabla {\cal Q}_h  u\cdot\bn\|^q_{L^q(\pT)}\Big)^{\frac{1}{q}} \Big(\sum_{T\in {\cal T}_h}\|  \lambda_b-\lambda_0\|^p_{L^p(\pT)}\Big)^{\frac{1}{p}}.
\end{split}
\end{equation}
For the term $\|\nabla  u\cdot\bn-\nabla {\cal Q}_h  u\cdot\bn\|^q_{L^q(\pT)}$, we have from the trace inequality \eqref{trace-inequality} that
\begin{equation}\label{te3}
\begin{split}
&\sum_{T\in {\cal T}_h}\| \nabla  u\cdot\bn-\nabla {\cal Q}_h  u\cdot\bn\|^q_{L^q(\pT)} \\
\leq &  \sum_{T\in {\cal T}_h}h_T^{-1}\Big(\|  \nabla  u-\nabla {\cal Q}_h  u\|^q_{L^q(T)}+h_T^q\|\nabla( \nabla  u-\nabla {\cal Q}_h  u)\|^q_{L^q(T)}\Big)\\
% \leq  &  \sum_{T\in {\cal T}_h}h_T^{q-1} \|\nabla(u-{\cal Q}_hu)\|^q_{L^q(T)}\\
 \leq & Ch^{(k-1)q-1}\|\nabla^{k}u\|^q_{L^q(\Omega)}.
\end{split}
\end{equation}
Substituting (\ref{te3}) into (\ref{e2}) gives
\begin{equation}\label{e4}
\begin{split}
&\Big|\sum_{T\in {\cal T}_h}\langle  \nabla {\cal Q}_h  u\cdot\bn - \nabla  u\cdot\bn , \lambda_b-\lambda_0\rangle_{\partial T} \Big|\\
\leq & Ch^{k}\|\nabla^{k}u\|_{L^q(\Omega)} \Big(\sum_{T\in {\cal T}_h}h_T^{1-2p} \|\lambda_b-\lambda_0\|^p_{L^p(\pT)}\Big)^{\frac{1}{p}}\\
\leq & C_3  \sum_{T\in {\cal T}_h}  h_T^{1-2p} \|\lambda_b-\lambda_0\|^p_{L^p(\pT)} + C_4 h^{qk}\|\nabla^{k}u\|^q_{L^q(\Omega)}.
\end{split}
\end{equation}
The third term can be analogously estimated by
\begin{equation}\label{e4:001}
\begin{split}
&\Big|\sum_{T\in {\cal T}_h}\langle (u-{\cal Q}_hu) \bbeta\cdot\bn, \lambda_0-\lambda_b\rangle_\pT\Big| \\
\leq &  C_5  \sum_{T\in {\cal T}_h}  h_T^{1-2p} \|\lambda_b-\lambda_0\|^p_{L^p(\pT)} + C_6 h^{qk}\|\nabla^{k}u\|^q_{L^q(\Omega)}.
\end{split}
\end{equation}
Substituting (\ref{e3}), (\ref{e4}), and \eqref{e4:001} into (\ref{EQ:April7:001}) gives
$$
(1-C_1-C_3-C_5) s(\lambda_h, \lambda_h)\leq Ch^{qk}\|\nabla^{k}u\|^q_{L^q(\Omega)},
$$
which leads to
$$
s(\lambda_h, \lambda_h)\leq Ch^{qk}\|\nabla^{k}u\|^q_{L^q(\Omega)}
$$
by choosing $C_i$ such that $1-C_1-C_3-C_5\ge C_0>0$. This completes the proof of the theorem.
\end{proof}

Consider the auxiliary problem that seeks $\phi$ such that
\begin{equation}\label{dual}
\begin{split}
-\Delta \phi - \bbeta\cdot\nabla\phi =\ &\psi,  \quad \text{in} \ \Omega,\\
\phi=\ &0, \quad \text{ on} \ \partial \Omega,
\end{split}
\end{equation}
where $\psi\in L^p(\Omega)$ is a given function. The problem \eqref{dual} is said to have the $W^{2,p}$-regularity if there exists a constant $C$ independent of $\psi$ satisfying
\begin{equation}\label{dual:regularity}
\|\phi\|_{2,p} \leq C \|\psi\|_{0,p}.
\end{equation}

The following is the main error estimate for the approximation $u_h$.

\begin{theorem} \label{theoestimate} Let $u$ and $(u_h;\lambda_h) \in M_h\times W_h^0$ be the exact solution of the second order elliptic problem \eqref{model} and its numerical solution arising from the PDWG scheme (\ref{32})-(\ref{2}). Assume the $W^{2,p}$-regularity estimate \eqref{dual:regularity} holds true for the auxiliary problem \eqref{dual}. Then the following $L^q$-error estimate holds true:
\begin{equation}\label{ErrorEstimate:Lq}
\|u^* - u_h\|_{L^q(\Omega)}\leq C h^{k}  \|\nabla^{k}u\|_{L^q(\Omega)},
\end{equation}
where $u^*$ is the solution of the following problem:
\begin{equation}\label{model:02}
\begin{split}
-\Delta u^* + \nabla\cdot(\bbeta u^*) &=Q_0 f, \quad \text{in} \ \Omega,\\
u^*&=Q_n g, \quad \text{on} \  \partial \Omega,
\end{split}
\end{equation}
 \end{theorem}

\begin{proof} For the solution $\phi$ of \eqref{dual}, by choosing $\sigma=Q_h \phi \in W_{h}^0$ in (\ref{32}) we obtain
$$
s(\lambda_h, Q_h \phi)+ (u_h, - \bbeta\cdot\nabla_w Q_h\phi - \Delta_w Q_h \phi) =  (f, Q_0\phi) - \langle g, Q_n(\nabla\phi\cdot\bn)\rangle_{\partial\Omega}.
$$
From \eqref{l} and \eqref{l:2} we have
$$
s(\lambda_h, Q_h \phi)+(u_h,  - {\cal Q}_h (\bbeta\cdot\nabla\phi + \Delta \phi)) =  (f, Q_0\phi)- \langle g, Q_n(\nabla\phi\cdot\bn)\rangle_{\partial\Omega},
$$
and further from \eqref{model:02} and \eqref{dual},
\begin{eqnarray*}
s(\lambda_h, Q_h \phi)+(u_h,  \psi) &=& (Q_0 f, \phi) - \langle Q_n g, \nabla\phi\cdot\bn\rangle_{\partial\Omega} \\
&=& (u^*,  -\Delta\phi - \bbeta\cdot\nabla\phi )\\
&=& (u^*, \psi) .
\end{eqnarray*}
It follows that
\begin{equation}\label{sf}
 (u^*-u_h,  \psi) = s(\lambda_h, Q_h \phi).
\end{equation}

To deal with the term $s(\lambda_h, Q_h \phi)$, from \eqref{EQ:local-stabilizer} we have
\begin{equation}\label{sab}
\begin{split}
&s(\lambda_h, Q_h \phi)=\sum_{T\in {\cal T}_h}h_T^{1-2p}\int_{\partial T} |\lambda_0-\lambda_b|^{p-1}sgn(\lambda_0-\lambda_b)(Q_0\phi-Q_b \phi)ds\\
&+ \sum_{T\in {\cal T}_h} h_T^{1-p}\int_{\partial T}  |\nabla \lambda_0 \cdot \bn-\lambda_n|^{p-1}sgn (\nabla \lambda_0 \cdot \bn-\lambda_n) (\nabla Q_0 \phi \cdot \bn-Q_n\phi)ds\\
&=I_1+I_2.
\end{split}
\end{equation}
For the term $I_1$, we use the Cauchy-Schwarz inequality and the trace inequality \eqref{trace-inequality} to obtain
\begin{equation}\label{I1}
\begin{split}
I_1&= \sum_{T\in {\cal T}_h}h_T^{1-2p}\int_{\partial T} |\lambda_0-\lambda_b|^{p-1}sgn(\lambda_0-\lambda_b)(Q_0\phi-Q_b \phi)ds
\\
&\leq \sum_{T\in {\cal T}_h}h_T^{1-2p} \Big(\int_{\partial T} |\lambda_0-\lambda_b|^{p}ds\Big)^{\frac{1}{q}} \Big(\int_{\partial T}(Q_0\phi-Q_b \phi)^pds\Big)^{\frac{1}{p}}\\
&\leq C\sum_{T\in {\cal T}_h}h_T^{1-2p}h_T^{(2p-1)/q}\Big(\int_{\partial T} h_T^{1-2p}|\lambda_0-\lambda_b|^{p}ds\Big)^{\frac{1}{q}} \\&\cdot \Big(h_T^{-1}\|Q_0\phi-\phi\|^p_{L^p(T)}+ h_T^{p-1}\|\nabla(Q_0\phi-\phi)\|^p_{L^p(T)}\Big)^{\frac{1}{p}}\\
&\leq Cs(\lambda_h,\lambda_h)^{\frac{1}{q}} \sum_{T\in {\cal T}_h} h_T^{\frac{1-2p}{p}} \Big(h_T^{2p-1}\|\nabla ^2\phi\|^p_{L^p(T)}\Big)^{\frac{1}{p}}\\
&\leq C s(\lambda_h,\lambda_h)^{\frac{1}{q}}\|\nabla ^2\phi\|_{L^p(\Omega)}. \\
\end{split}
\end{equation}
Similarly, for $I_2$ we again use the Cauchy-Schwarz inequality and the trace inequality \eqref{trace-inequality} to obtain
\begin{equation}\label{I2}
\begin{split}
I_2&= \sum_{T\in {\cal T}_h} h_T^{1-p}\int_{\partial T}  |\nabla \lambda_0 \cdot \bn-\lambda_n|^{p-1}sgn (\nabla \lambda_0 \cdot \bn-\lambda_n) (\nabla Q_0 \phi \cdot \bn-Q_n\phi)ds
\\
&\leq \sum_{T\in {\cal T}_h}h_T^{1-p} \Big(\int_{\partial T} |\nabla \lambda_0 \cdot \bn-\lambda_n|^{p} ds\Big)^{\frac{1}{q}} \Big(\int_{\partial T}(\nabla Q_0 \phi \cdot \bn-Q_n\phi)^pds\Big)^{\frac{1}{p}}\\
&\leq C\sum_{T\in {\cal T}_h}h_T^{1-p}h_T^{(p-1)/q}\Big(\int_{\partial T} h_T^{1-p}|\nabla \lambda_0 \cdot \bn-\lambda_n|^{p}ds\Big)^{\frac{1}{q}} \\
&\ \ \Big(h_T^{-1}\|\nabla Q_0 \phi - \nabla \phi\|^p_{L^p(T)}+ h_T^{p-1}\|\nabla(\nabla Q_0 \phi -\nabla\phi)\|^p_{L^p(T)}\Big)^{\frac{1}{p}}\\
&\leq C s(\lambda_h,\lambda_h)^{\frac{1}{q}}\|\nabla ^2\phi\|_{L^p(\Omega)}. \\
\end{split}
\end{equation}
Substituting (\ref{I1}) and \eqref{I2} into \eqref{sab} and using \eqref{sestimate} and \eqref{dual:regularity} yields
\begin{equation*}
\begin{split}
\Big|s(\lambda_h, Q_h \phi)\Big|\leq & Cs(\lambda_h,\lambda_h)^{\frac{1}{q}}\|\nabla ^2\phi\|_{L^p(\Omega)}\\
\leq & C \Big(h^{qk}\|\nabla^{k}u\|^q_{L^q(\Omega)}\Big)^{\frac{1}{q}}\|\nabla ^2\phi\|_{L^p(\Omega)}\\
\leq & C h^{k}  \|\nabla^{k}u\|_{L^q(\Omega)}\|\nabla ^2\phi\|_{L^p(\Omega)}\\
\leq & C h^{k}  \|\nabla^{k}u\|_{L^q(\Omega)}\|\psi\|_{L^p(\Omega)},
\end{split}
\end{equation*}
which, together with (\ref{sf}), leads to
 \begin{equation*}
\begin{split}
\big|(u^* - u_h,  \psi)\Big|\leq C h^{k}  \|\nabla^{k}u\|_{L^q(\Omega)}\|\psi\|_{L^p(\Omega)},
\end{split}
\end{equation*}
so that
$$
\|u^* - u_h\|_{L^q(\Omega)}\leq C h^{k}  \|\nabla^{k}u\|_{L^q(\Omega)}.
$$
This completes the proof of the theorem.
\end{proof}

\section{Error Estimates for the Dual Variable}\label{Section:ES-2}
 In this section  we shall establish some error estimates for the dual variable $\lambda_h$
 in $W^{1,p}$ and $L^p$. To this end, let $\varphi$ be the solution of the following auxiliary problem
\begin{eqnarray}\label{dual-problem:new}
\begin{aligned}
    -\triangle \varphi + \bbeta\cdot\nabla\varphi &=\theta, \ \ \mbox{ in} \ \Omega, \\
   \varphi&=0,\ \  \mbox{ on} \ \partial\Omega,
\end{aligned}
\end{eqnarray}
where $\theta$ is a given function in $L^q(\Omega)$. Assume the dual problem (\ref{dual-problem:new}) has the $W^{2,q}$-regularity in the sense that there exists a constant $C$ such that
\begin{equation}\label{regularity:1}
\|\varphi\|_{2,q}\leq C \|\theta\|_{0,q}.
\end{equation}
%It was proved in \cite{WW-mathcomp} that

From \eqref{Loperator1-2} and the usual integration by parts we have
\begin{equation*}%\label{e:11}
\begin{split}
 & (\triangle_w v, \varphi) =  (\triangle_w v, {\cal Q}_h\varphi) \\
= & \sum_{T\in\T_h} ( \triangle v_0, {\cal Q}_h\varphi)_T + \langle v_0-v_b, \nabla ({\cal Q}_h\varphi)\cdot\bn \rangle_\pT + \langle v_n - \nabla v_0 \cdot\bn, {\cal Q}_h\varphi\rangle_\pT\\
= & \sum_{T\in\T_h} ( \triangle v_0, \varphi)_T + \langle v_0-v_b, \nabla ({\cal Q}_h\varphi)\cdot\bn \rangle_\pT + \langle v_n - \nabla v_0 \cdot\bn, {\cal Q}_h\varphi\rangle_\pT\\
= & \sum_{T\in\T_h} (v_0, \triangle\varphi)_T + \langle v_0-v_b, \nabla ({\cal Q}_h\varphi -\varphi)\cdot\bn \rangle_\pT + \langle v_n - \nabla v_0 \cdot\bn, {\cal Q}_h\varphi -\varphi\rangle_\pT
\end{split}
\end{equation*}
so that
\begin{equation}\label{e:11}
\begin{split}
      (v_0, \triangle\varphi)=&\ (\triangle_w v, \varphi) + \sum_{T\in\T_h}\langle
  \varphi-{\cal Q}_h \varphi,  v_n-\nabla v_0\cdot{\bf n}\rangle_\pT\\
&   + \sum_{T\in\T_h}\langle v_0-v_b,\nabla(\varphi-{\cal Q}_h \varphi)\cdot {\bf n}\rangle_\pT.
\end{split}
\end{equation}
Analogously, from \eqref{Gradient-2} we have (note that $s=k-1$ so that $\nabla v_0\in [P_s(T)]^d$)
\begin{equation}\label{e:11:02}
\begin{split}
(v_0, \bbeta\cdot\nabla\varphi)_T&= - (\nabla v_0, \bbeta\varphi)_T +\langle v_0, \varphi \bbeta\cdot\bn\rangle_\pT \\
&=- (\nabla v_0, \bbeta{\cal Q}_h\varphi)_T +\langle v_0, \varphi \bbeta\cdot\bn\rangle_\pT\\
&=- (\bbeta\cdot\nabla_w v, \varphi)_T +\langle v_b-v_0, {\cal Q}_h\varphi\bbeta\cdot\bn\rangle_\pT
+\langle v_0, \varphi \bbeta\cdot\bn\rangle_\pT.\\
\end{split}
\end{equation}
Summing \eqref{e:11:02} over all $T\in\T_h$ yields
\begin{equation}\label{e:11:05}
\begin{split}
(v_0, \bbeta\cdot\nabla\varphi)
=&- (\bbeta\cdot\nabla_w v, \varphi)+ \sum_{T\in\T_h}\langle v_0-v_b, (I-{\cal Q}_h)\varphi\bbeta\cdot\bn\rangle_\pT.
\end{split}
\end{equation}

We have the following error estimates for the variable $\lambda_0$.

\begin{theorem}\label{theo:1}
 Let $u$ and $(u_h;\lambda_h) \in M_h\times W_h^0$ be the  solutions of  \eqref{model} and (\ref{32})-(\ref{2}), respectively. Assume that the dual problem \eqref{dual-problem:new}
has the $W^{2,q}(\Omega)$ regularity with the a priori estimate \eqref{regularity:1}.
Then the following estimate hold true:
\begin{equation}\label{es:2}
   \|\lambda_0\|_{0,p}\le
\left\{
\begin{array}{ll}
C h^{q}\|u\|^{q/p}_{1,q}, &\quad k=1,\\
C h^{(q-1)k+2}\|u\|^{q/p}_{k,q}, &\quad k>1.
\end{array}
\right.
\end{equation}
\end{theorem}
\begin{proof}
For any given function $\theta\in L^q(\Omega)$, let $\varphi$ be the solution of \eqref{dual-problem:new}. From  \eqref{e:11} and \eqref{e:11:05} we have
 \begin{eqnarray*}
 (\theta, \lambda_0)&=&(\lambda_0,-\triangle\varphi+\bbeta\cdot\nabla\varphi)\\
&=&(-\triangle_w \lambda_h - \bbeta\nabla_w\lambda_h, \varphi) \\
&& + \sum_{T\in\T_h}\langle
  \varphi-{\cal Q}_h \varphi,  \nabla \lambda_0\cdot{\bf n}-\lambda_n\rangle_\pT +\sum_{T\in\T_h}\langle \lambda_b-\lambda_0,\nabla(\varphi-{\cal Q}_h \varphi)\cdot {\bf n}\rangle_\pT \\
&& + \sum_{T\in\T_h}\langle \lambda_0-\lambda_b, (I-{\cal Q}_h)\varphi\bbeta\cdot\bn\rangle_\pT\\
  &=&I_1+I_2+I_3+I_4.
 \end{eqnarray*}
We next estimate $I_i,\ 1\le i\le 4$,  respectively.   In light of \eqref{2}, we have
\begin{eqnarray*}
|I_1| &=& |(\triangle_w \lambda_h + \bbeta\cdot\nabla_w\lambda_h, \varphi)| \\
&=& |(\triangle_w \lambda_h + \bbeta\cdot\nabla_w\lambda_h, \varphi-{\cal Q}_h \varphi)| \\
&=& 0.
 \end{eqnarray*}
As to $I_2$, we have from the Cauchy-Schwarz inequality and the trace inequality \eqref{trace-inequality},
\begin{eqnarray*}
|I_2|&=& \left|  \sum_{T\in\T_h}\langle
  \varphi-{\cal Q}_h \varphi,  \nabla \lambda_0\cdot{\bf n}-\lambda_n\rangle_\pT    \right|\\
&\leq & \left(\sum_{T\in\T_h} h_T^{1-p} \| \nabla \lambda_0\cdot{\bf n}-\lambda_n\|^p_{p,\pT}\right)^{\frac1p} \left( \sum_{T\in\T_h} h_T \|\varphi-{\cal Q}_h \varphi \|^q_{q,\pT}\right)^{\frac1q}\\
&\lesssim & s(\lambda_h,\lambda_h)^{\frac 1p}(\|\varphi-{\cal Q}_h \varphi\|_{0,q}+{h\color{red}^q}\|\nabla(\varphi-{\cal Q}_h \varphi)\|_{0,q})\\
 &\lesssim & h^{m+1} s(\lambda_h,\lambda_h)^{\frac 1p}\|\varphi\|_{m+1,q}.
\end{eqnarray*}
 Following the same argument, there holds
\begin{eqnarray*}
|I_3| \lesssim h^{m+1} s(\lambda_h,\lambda_h)^{\frac 1p}\|\varphi\|_{m+1,q},\quad |I_4| \lesssim h^{m+2} s(\lambda_h,\lambda_h)^{\frac 1p}\|\varphi\|_{m+1,q}.
\end{eqnarray*}
Here $m=1$ for $k>1$ and $m=0$ for $k=1$. It is so because there holds $P_1(T)\subset M_h(T)$ for $k>1$ and $M_h(T)$ consists of only piecewise constants for $k=1$.

By combining all the estimates for $I_i, i\le 4$ and the estimate \eqref{sestimate} in Theorem \ref{theoestimate-1} we arrive at
\begin{equation}\label{eq:12}
|(\lambda_0, \theta)| \lesssim h^{k(q-1)+m+1}\|u\|^{q/p}_{k,q}\|\varphi\|_{m+1,q},\ \ 0\le m\le 1.
\end{equation}
The estimate \eqref{es:2} then follows from the $W^{2,q}(\Omega)$-regularity  \eqref{regularity:1} with $m=0$  and $m=1$, respectively. This completes the proof.
\end{proof}

\begin{theorem}\label{theo:2}
Let $u$ and $(u_h;\lambda_h) \in M_h\times W_h^0$ be the  solutions of  \eqref{model} and (\ref{32})-(\ref{2}), respectively.
Assume that the dual problem \eqref{dual-problem:new} has the $W^{2,q}(\Omega)$-regularity with the a priori estimate \eqref{regularity:1}. Then
\begin{equation}\label{est:3}
   \|\nabla \lambda_0\|_{0,p}\le C h^{k(q-1)+1}\|u\|^{q/p}_{k,q}.
\end{equation}
\end{theorem}
\begin{proof} For any given function ${\bf \eta}\in [C^1(\Omega)]^d$ with $\eta=0$ on ${\cal E}_h$, let $  \varphi$ be the solution of the dual problem \eqref{dual-problem:new} with
$\theta=-\nabla\cdot \eta$. It is easy to see that
\begin{eqnarray*}
   (\nabla \lambda_0, {\bf \eta})=-(\lambda_0, \nabla\cdot {\bf \eta})=(\lambda_0, \theta).
\end{eqnarray*}
  In light of \eqref{eq:12}, we get
  \[
      |(\nabla\lambda_0, {\bf \eta}) |\lesssim h^{k(q-1)+1}\|u\|^{q/p}_{k,q}\|\varphi\|_{1,q}
      \lesssim  h^{k(q-1)+1}\|u\|^{q/p}_{k,q}\|\eta\|_{0,q}
  \]
 with $\varphi$ the solution of \eqref{dual-problem:new}.
 As the set of all such $\eta$ is dense in $L^q(\Omega)$, hence
  \[
     \|\nabla \lambda_0\|_{0,p}\lesssim h^{k(q-1)+1}\|u\|^{q/p}_{k,q}.
  \]
     This completes the proof.
\end{proof}

\section{Numerical Results}\label{Section:NR}
Our numerical experiments are based on the PDWG algorithm \eqref{32}-\eqref{2} with $k=1,2$ for the finite element spaces $W_h$ and $M_h$ defined in \eqref{wh}-\eqref{mh}. The system of nonlinear equations \eqref{32}-\eqref{2} is solved by using an iterative scheme similar to that for the $L^1$ minimization problem in \cite{Vogel}. Specifically, given an approximation $(u_h^m,\lambda_h^m)$ at step $m$, the scheme shall compute a new approximate solution $(u_h^{m+1},\lambda_h^{m+1})\in M_h\times W_h^0$ such that
\begin{eqnarray*}
\tilde s(\lambda^{m+1}_h, \sigma)+b(u^{m+1}_h, \sigma)&=& (f, \sigma_0) - \langle g, \sigma_n  \rangle_{\partial \Omega}, \qquad \forall \sigma\in W_{h}^0,\\
b(v, \lambda^{m+1}_h)&=&0,\qquad \qquad \qquad   \qquad \qquad \forall v\in M_h,
\end{eqnarray*}
 where
 \begin{eqnarray*}
   \tilde s(\lambda^{m+1}_h, \sigma)&=&\sum_{T\in{\cal T}_h}h_T^{1-2p}\int_{\partial T} (|\lambda^m_0-\lambda^m_b|+\epsilon)^{p-2}(\lambda^{m+1}_0-\lambda^{m+1}_b)(\sigma_0-\sigma_b)ds\\
 &+ & \sum_{T\in{\cal T}_h} h_T^{1-p}\int_{\partial T}  (|\nabla\lambda^m_0 \cdot \bn-\lambda^m_n|+\epsilon)^{p-2}(\nabla \lambda^{m+1}_0 \cdot \bn-\lambda^{m+1}_n) (\nabla \sigma_0 \cdot \bn-\sigma_n)ds.
\end{eqnarray*}
Here $\epsilon$ is a small, but positive constant. All the numerical results are obtained with $\epsilon=10^{-3}$ if not otherwise stated. Various approximation errors are computed for $u_h$ and $\lambda_h$, including the $L^q$ error for $e_h:=u-u_h$, and the $W^{2,p},\ W^{1,p},$ and $ L^p$ errors for $\lambda_h$. The finite element partition ${\cal T}_h$ is obtained through a successive refinement of a coarse triangulation of the domain, by
dividing each coarse element into four congruent sub-elements by connecting the midpoints of the three edges of the triangle. The right-hand side function, the boundary condition are calculated from the exact solution.

\begin{example}\label{Example1}
The domain in the model problem \eqref{model} is given by $\Omega=(0,1)^2$. Vanishing convection $\bbeta=0$ is considered in this test problem. The functions $f$ and $g$ are chosen so that the exact solution is given by $u=\sin(x)\sin(y)$.
\end{example}

Tables \ref{1}-\ref{3} illustrate the approximation error and the rate of convergence for the primal variable $u_h$ and the dual variable $\lambda_0$ with $k=1,2$ and $p=1,\cdots, 5$. For $p>1$, one observes a convergence rate of $\mathcal{O}(h^{s+1})$ for the error $\|e_h\|_{0,q}$ with $s=k-1, k-2$, which is consistent with the theory shown in Theorem \ref{theoestimate}. For $p=1$ and $k=2$, one observes a convergence rate of $\mathcal{O}(h^{s+1})$ for the error $\|e_h\|_{0,\infty}$ - an optimal order of convergence in the maximum norm. For the dual variable approximation $\lambda_0$, the tables suggest the following rates of convergence:
$\mathcal{O}(h^{p})$ for $\|\lambda_0\|_{2,p}$, $\mathcal{O}(h^{p+1})$ for $\|\lambda_0\|_{1,p}$, and $\mathcal{O}(h^{p+2})$ for $\|\lambda_0\|_{0,p}$ with $k=2,s=k-1$. For the case of $s=0$ with $k=1,2$, one observes a convergence of $\mathcal{O}(h^{p-1})$ for $\|\lambda_0\|_{2,p}$ and $\mathcal{O}(h^{p})$ for $\|\lambda_0\|_{m,p}, m=0,1$. In other words, this numerical experiment suggests $\|\lambda_0\|_{m,p} = \mathcal{O}(h^{p+2-m}),m\le 2,$ when the approximating space $W_h(T)$ contains linear functions on each element. The convergence is reduced to $\mathcal{O}(h^{p+1-\max(1,m)})$ if only piecewise constant functions are seen in $W_h(T)$.
We emphasize that the numerical dual variable $\lambda_0$ outperforms the theory predicted in Theorems \ref{theo:1}-\ref{theo:2} with rates depending on $p$. This $p$-dependence of the convergence remains mysterious to the authors.

%Tables \ref{1}-\ref{3} also include the approximation errors and convergence rates for the $L^1$-PDWG method, i.e., $p=1$. Like the case of $p>1$, the numerical results suggest a rate of convergence of $\mathcal{O}(h^{s+1})$ for the variable $u_h$. But for the dual variable $\lambda_h$,  the convergence rate varies according to the  value of $s$, i.e.,  the construction of $M_h$. It is interesting to see that, for $s=0$,  no convergence can be seen  for $\|\lambda_0\|_{m,p}, \ m\le 2$.

 \begin{table}[h] %[htbp]
\caption{Numerical error and rate of convergence for the $L^p$-PDWG method with $k=2, s=k-1$ for Example 1.}\label{1} \centering
\tiny
\begin{threeparttable}
        \begin{tabular}{c |c c c c c c c cc }
        \hline
    & h&   $\|\lambda_0\|_{2,p}$ & rate& $\|\lambda_0\|_{1,p}$  & rate& $\|\lambda_0\|_{0,p}$ & rate& $\|e_h\|_{0,q}$ &rate   \\
       \hline \cline{2-10}

            &3.54e-01  & 2.05e-03  &     --& 2.58e-04     &      --   &1.50e-04      &   --& 1.48e-02    &        --\\
             &1.77e-01 &  9.95e-04 &  1.04&   4.89e-05 &  2.40 &   3.60e-05 &  2.06&  3.92e-03&   1.92 \\
$p=1$  & 8.84e-02 &  4.92e-04 &  1.02&   1.05e-05&   2.22&  8.85e-06 &  2.02&   1.00e-03 &  1.97 \\
           & 4.42e-02  & 2.45e-04  & 1.06 &  2.42e-06 &  2.11& 2.19e-06   &2.01& 2.53e-04  & 1.98 \\
           & 2.21e-02 &  1.22e-04&   1.00&  5.81e-07  & 2.06&  5.46e-07   &2.01&   6.37e-05&   1.99 \\
 \hline
            & 3.54e-01 &  1.05e-02 &   --    &1.82e-03    &    --    &   4.75e-04  &    --    &4.01e-03   &     -- \\
            & 1.77e-01 &  2.54e-03 &  2.05&   1.79e-04 &  3.35 &  2.95e-05 &  4.01 &  9.96e-04&   2.01 \\
$p=2$  & 8.84e-02 &  6.26e-04  & 2.02 &  1.96e-05 &  3.18  &  1.84e-06&   4.00 &  2.48e-04 &  2.00 \\
            & 4.42e-02 &  1.55e-04 &  2.01&  2.30e-06&   3.09 &  1.15e-07 &  4.00 &  6.20e-05  & 2.00 \\
           & 2.21e-02 &  3.87e-05 &  2.00&  2.79e-07 &  3.05  & 7.20e-09 &  4.00 &  1.55e-05 &  2.00  \\
 \hline
            &  3.54e-01  &  7.36e-02  &      --     & 1.12e-02    &        -   &1.35e-03    &  --    &   3.22e-03  &  --  \\
            &   1.77e-01  &   9.05e-03  &  3.02  &  5.52e-04  &  4.34   & 2.65e-05  &  5.67  &  7.60e-04 &  2.08\\
$p=3$  &   8.84e-02   & 1.03e-03  &  3.14  &  2.76e-05  &  4.32  & 5.77e-07  &  5.52   & 1.77e-04 &  2.11\\
            &   4.42e-02  &   1.12e-04   & 3.19  &  1.39e-06   & 4.31   & 1.42e-08   & 5.35   & 4.04e-05 &  2.13\\
           &   2.21e-02     &1.21e-05  &  3.22  &  7.46e-08  &  4.22 &  3.85e-10   & 5.12   & 9.35e-06  & 2.11\\
     \hline
            & 3.54e-01   &  4.27e-01  &     --   &6.56e-02     &     -- &  3.00e-03  &        -- & 2.96e-03   &         -- \\
            & 1.77e-01    &  2.73e-02 &  3.97 &  1.50e-03 &  5.45 &  4.83e-05 &  5.96&   6.05e-04&   2.29 \\
$p=4$   &  8.84e-02  &  1.20e-03 &  4.51  & 3.65e-05 &  5.36&   7.17e-07&   6.07&   1.37e-04 &  2.14 \\
            &4.42e-02    & 6.39e-05 &  4.23 &  1.23e-06  & 4.90  & 1.17e-08  & 5.94  & 3.36e-05  & 2.03 \\
           & 2.21e-02    &  3.83e-06 &  4.06&  4.04e-08&   4.93  & 1.86e-10  & 5.97&  8.37e-06 &  2.00 \\
     \hline
            &3.54e-01 &  2.73e-02 &  -- &  4.79e-02      & --     & 5.48e-03  &--&   3.91e-03   &--\\
 $p=5$ &  1.77e-01&   6.50e-02 & -1.25&   4.14e-03&   3.53 &  1.51e-04&   5.19 &  5.41e-04 &  2.85\\
            & 8.84-02&   1.84e-03.  &   5.14&   8.16e-05 &  5.66  &1.56e-06 &  6.60 &  1.31e-04 &  2.04\\
           & 4.42e-02&   5.75e-05 &  5.00&  1.32e-06 &  5.94 & 1.25e-08  & 6.97 &  3.28e-05  & 2.00\\
  \hline
 \end{tabular}
 \end{threeparttable}
\end{table}

 \begin{table}[h] %[htbp]
\caption{Numerical error and rate of convergence for the $L^p$-PDWG method with $k=2, s=k-2$ for Example 1.}\label{ex22}\centering
\tiny
\begin{threeparttable}
        \begin{tabular}{c |c c c c c c c cc }
        \hline
   & h&   $\|\lambda_0\|_{2,p}$ & rate& $\|\lambda_0\|_{1,p}$  & rate& $\|\lambda_0\|_{0,p}$ & rate& $\|e_h\|_{0,q}$ &rate   \\
       \hline \cline{2-10}

          &   3.54e-01.  &  4.27e-02   &    --  &1.23e-01 &  --     &  3.23e-03  &     --  & 1.02e-01 &     --     \\
           & 1.77e-01   &  4.16e-02 &  0.04  & 1.33e-01&  -0.11&   2.08e-03&   0.64&   5.46e-02&   0.90\\
 $p=1$ &  8.84e-02 & 4.02e-02 &  0.05 &  1.40e-01  & -0.06 &  1.60e-03&  0.38 &  2.82e-02&   0.95\\
          &  4.42e-02   &  3.90e-02 &  0.04 & 1.43e-01 & -0.04 &  1.44e-03 &  0.14 &  1.45e-02  & 0.96\\
           &  2.21e-02 &  3.85e-02 &  0.02 &  1.45e-01 & -0.02&  1.40e-03&   0.04 &  7.36e-03 &  0.97\\
           \hline
           &3.54e-01  &  8.54e-02    &         --   & 2.75e-01   &   --   & 1.32e-02     &    -- & 1.04e-01  &          --\\
            & 1.77e-01 &   3.70e-02 &   1.21  &  7.81e-02  &   1.82 &   2.37e-03  &  2.47  &  5.44e-02  &  0.93\\
 $p=2$ & 8.84e-02 &   1.60e-02  &  1.21  &  2.08e-02  &  1.91  &  4.95e-04  & 2.26  &  2.78e-02  &  0.97\\
           &  4.42e-02  &  7.12e-03  &  1.17  &  5.37e-03  &  1.95  &  1.16e-04   & 2.09  &  1.41e-02  &  0.98\\
          &   2.21e-02  &  3.28e-03  &  1.12  &  1.36e-03  &  1.98  &  2.88e-05   & 2.01 &  7.07e-03 &   0.99\\
\hline
             & 3.54e-01 &   3.12e-01 &         --   & 5.23e-01    &      --  & 3.21e-02  &      --  & 1.00e-01   &        --  \\
              &1.77e-01 &   1.38e-01  &  1.18  &  6.50e-02  &  3.01 &   2.50e-03 &   3.68   & 5.31e-02   &  0.92\\
 $p=3$   &8.84e-02 &   2.89e-02 &   2.26   &  4.56e-03 &   3.83  &  1.36e-04   & 4.20  &  2.71e-02  &0.97\\
             & 4.42e-02   & 5.42e-03 &   2.42 &  3.225e-04  &  3.82  &  1.70e-05  &  3.01    &1.35e-02 &  1.01\\
            &2.21e-02   & 1.01e-03 &   2.43  &  2.47e-05 &   3.71  &  2.29e-06  &  2.89   & 6.66e-03   &   1.02\\
\hline
         &3.54e-01&   2.36e-01 &       -   &1.49e-01   &     --  & 2.33e-02       &   -- & 1.12e-01 &    --\\
         & 1.77e-01&   3.08e-01&  -3.84 &  6.84e-02 &  1.12 &  3.75e-03 &  2.63 &  5.02e-02  & 1.16\\
$p=4$ &8.84e-02&   1.80e-02 &  4.10&   2.06e-03&   5.05&   2.52e-04  & 3.90 &  2.48e-02  & 1.02\\
          &4.412e-02 &  7.86e-04 &  4.51&   7.88e-05 &  4.71&   1.47e-05 &  4.10 &  1.28e-02  & 0.96\\
         & 2.21e-02 &  5.87e-05&   3.74&   4.13e-06 &  4.25 &   9.02e-07 &  4.03  & 6.46e-03 &  0.98\\
\hline
         &3.54e-01 &  5.21e-03 &       -- &  1.23e-02     &     --&   3.61e-03    &    -   &4.38e-01  &          --\\
           & 1.77e-01&   2.57e-01 & -5.62&   5.13e-02  &-2.06&   2.80e-03&   0.37  & 5.10e-02  & 3.10\\
 $p=5$  & 8.84e-02 &  1.32e-02&   4.28&   2.16e-03 &  4.57&   4.07e-04&   2.78 &  2.38e-02  & 1.10\\
            &  4.42e-02 &  6.97e-04 &  4.24 &  6.62e-05 &  5.03&   1.41e-05 &  4.85  & 1.23e-02  & 0.95\\
           &   2.21e-02 &  3.87e-05  & 4.17  & 2.02e-06 &  5.03&   4.41e-07&   5.00  & 6.24e-03 &  0.98\\
    \hline
 \end{tabular}
 \end{threeparttable}
\end{table}

 \begin{table}[h] %[htbp]
\caption{Numerical error and rate of convergence for the $L^p$-PDWG method with $k=1, s=k-1$ for Example 1.}\label{3}\centering
\tiny
\begin{threeparttable}
        \begin{tabular}{c | c c c c c c c }
        \hline
    & h&   $\|\lambda_0\|_{1,p}$  & rate& $\|\lambda_0\|_{0,p}$ & rate& $\|e_h\|_{0,q}$ &rate   \\
       \hline \cline{2-8}
           &3.54e-01   &        1.19e-01 &  --    & 3.11e-03  &   --   &1.12e-01      &     --\\
           & 1.77e-01   &      1.31e-01  & 0.14  & 1.77e-03 &  0.81 &  6.08e-02 &  0.89\\
 $p=1$ &  8.84e-02 &      1.39e-01&  -0.08&   1.03e-03&  0.78&   3.60e-02  & 0.76\\
           &4.42e-02 &         1.43e-01 & -0.04 &  8.77e-04  & 0.23  & 2.22e-02 &  0.70\\
         & 2.21e-02  &        1.45e-01 & --0.02&  1.00e-03 & -0.20  & 1.34e-02  & 0.72\\
\hline
   & 3.54e-01  &          2.50e-01  &    --&1.14e-02   &      --  & 3.45e-02   &           \\
  & 1.77e-01  &          7.44e-02  & 1.75 & 2.07e-03  & 2.46&   1.75e-02 &  0.98\\
$p=2$  & 8.84e-02  &             2.04e-02&   1.87 &  4.83e-04&   2.10 &  8.46e-03&   1.05\\
  & 4.42e-02  &            5.32e-03&   1.94&  1.22e-04 &  1.98 &  4.08e-03 &  1.05\\
   &2.21e-02 &           1.36e-03 &  1.97 &  3.10e-05  & 1.98 & 2.01e-03  & 1.03\\

\hline
             & 3.54e-01 &     4.68e-01  &    --  & 2.89e-02     &      --   &3.19e-02    &        --\\
             & 1.77e-01   &    5.87e-02  & 2.99 &  2.57e-03  &     3.49&   1.54e-02 &  1.06\\
$p=3$   & 8.84e-02  &     4.18e-03&   3.81 &  1.14e-04 &   4.50 & 7.48e-03  & 1.04\\
             & 4.42e-02 &      2.81e-04 &  3.90  &8.05e-06  & 3.82&  3.73e-03  & 1.00\\
            & 2.21e-02&       1.96e-05 &  3.84 &  1.27e-06 &  2.66  &  1.86e-03  & 1.00\\

\hline
              & 3.54-01 &     1.41e-01  &       --   &1.67e-02  &   --  & 3.38e-02  &          --\\
             & 1.77e-01 &     5.22e-02  & 1.44&   2.96e-03&   2.50&   1.56e-02  & 1.12\\
 $p=4$  & 8.84e-02 &      1.54e-03 &  5.08&   1.90e-04&   3.96&   7.52e-03 &  1.05\\
             & 4.42e-02 &       6.88e-05 &  4.48 &  1.37e-05 &  3.79 &  3.75e-03   &1.00\\
          &  2.21e-02 &        3.99e-06  & 4.11 &  8.87e-07 &  3.95 & 1.88e-03 &  1.00\\

\hline
           &3.54e-01       &          2.62e-02  &    --& 4.16e-03    &        --  & 6.20e-02 &              --\\
           &  1.77e-01    &          4.03e-02 & -0.62 &   3.01e-03 &  0.47 &  1.70e-02 &  1.87\\
$p=5$ &   8.84e-02  &           1.88e-03 &  4.42&   3.91e-04 &  2.94&   7.76e-03&   1.13\\
           &   4.42e-02 &            6.35e-05 &  4.89 &  1.40e-05&   4.80&   3.86e-03 &  1.01\\
          &  2.21e-02  &          1.98e-06 &  5.00 &  4.40e-07  & 4.99 &  1.93e-03  & 1.00\\
      \hline
 \end{tabular}
 \end{threeparttable}
\end{table}

\begin{example}\label{Example2} The domain in this test case is given by $\Omega=(0,1)^2$. The convection term has the following form $\bbeta=(-y,x)$.
The functions $f$ and $g$ are chosen such that the exact solution to the elliptic problem is given by $u= \frac 12 \sin(x+y)+\cos(x-y)+\frac 32$.
\end{example}

Tables \ref{ex3:1}-\ref{ex3:2} show the approximation error and rates of convergence for the primal variable $u_h$ and the dual variable $\lambda_0$ with $k=1,2$ and $p=1,\cdots, 5$. As the model problem contains a non-trivial convection term, the space $W_h$ is taken as \eqref{mh} with $s=k-1$. From Tables \ref{ex3:1}-\ref{ex3:2} we observe the same convergence phenomenon as that for the purely diffusive equation in Example 1. More precisely, it is observed that the error $\|e_h\|_{0,q}$ converges to zero at the rate of $\mathcal{O}(h^{s+1})$ for $p>1$, which is consistent with the theory developed in Theorem \ref{theoestimate}. For $\lambda_0$, it appears that the convergence is dependent upon the value of $p$: $\mathcal{O}(h^{p+2-m})$ for $k=2$ and $\mathcal{O}(h^{p+2-\max(1,m)})$ for $k=1$ in the metric $\|\lambda_0\|_{m,p}, m\le 2$. In other words, the construction of the finite element space $M_h$ has effect on the convergence of the dual variable $\lambda_0$. Like in Example \ref{Example1}, the numerical convergence for the dual variable $\lambda_0$ is faster than the theory predicted in Theorems \ref{theo:1}-\ref{theo:2} with non-trivial convection terms in the model equation.

Tables \ref{ex3:1}-\ref{ex3:2} also show the approximation error and convergence rates for the $L^1$-PDWG method, i.e., $p=1$. For the case of $k=2$ and $s=1$, we see an optimal order of convergence in the maximum norm for the primal variable. Like Example \ref{Example1}, the convergence for the dual variable $\lambda_0$ varies with respect to the choice of $M_h$. For $s=1$, Table \ref{ex3:1} shows a convergence rate of $\mathcal{O}(h^{p+2-\max(1,m)})$ for the error $\|\lambda_0\|_{m,p}$ with $m\le 2$. For $s=0$, no convergence is seen from Table \ref{ex3:2} for $\|\lambda_0\|_{m,p}, \ m\le 1$.

\begin{table}[h]%[htbp]
\caption{Numerical error and rate of convergence for the $L^p$-PDWG method with $k=2, s=k-1$ for Example 2.}\label{ex3:1} \centering
\tiny
\begin{threeparttable}
        \begin{tabular}{c |c c c c c c c cc }
        \hline
   & h&   $\|\lambda_0\|_{2,p}$ & rate& $\|\lambda_0\|_{1,p}$  & rate& $\|\lambda_0\|_{0,p}$ & rate& $\|e_h\|_{0,q}$ &rate   \\
       \hline \cline{2-10}
        &3.54e-01 &  4.87e-03 &     --  & 8.18e-04 &   --  & 8.78e-04  &       --  & 2.47e-02  &    --      \\
        &1.77e-01&   2.45e-03 & 0.99 &  2.18e-04 &  1.91  & 2.17e-04 &  2.01 &  6.16e-03  & 2.00\\
 $p=1$  &8.84e-02&   1.23e-03 &  1.00 &  5.58e-05 &  1.96&  5.40e-05 &  2.01 &  1.58e-03 &  1.96\\
        &4.42e-02 &  6.14e-04 &  1.00 &  1.41e-05 &  1.98 & 1.35e-05 &  2.00 &  3.98e-04 &  1.99\\
        & 2.21e-02 &  3.07e-04 &  1.00 &  3.54e-06 &  1.99&  3.36e-06 &  2.00 &  1.00e-04 &  1.99\\

  \hline
         & 3.54e-01&   2.20e-02 &     -- &  2.26e-03 &     -- &  2.17e-03 &     -- &  5.73e-03 &    --     \\
         & 1.77e-01&   5.46e-03 &  2.01 &  2.37e-04  & 3.26  & 1.36e-04&   4.00 &  1.41e-03 &  2.02\\
  $p=2$  & 8.84e-02&   1.36e-03&   2.00 &  2.74e-05 &  3.11 &  8.49e-06&   4.00&   3.52e-04 &  2.01\\
         &4.42e-02&   3.41e-04&   2.00 &  3.33e-06 &  3.04  & 5.30e-07 &  4.00 &  8.80e-05  & 2.00\\
         &2.21e-02&   8.51e-05&   2.00  & 4.11e-07&   3.02 & 3.31e-08&   4.00 &  2.20e-05  & 2.00\\
\hline

        & 3.54e-01&   1.33e-01 &     -- &  1.40e-02&   -- &  5.21e-03 &        -- &  4.70e-03 &   --     \\
        & 1.77e-01 &  1.80e-02 &  2.89 &  8.77e-04&   3.99 &  9.80e-05&   5.73 &  1.16e-03 &  2.02\\
  $p=3$ &8.84e-02 &  2.17e-03 &  3.05 &  5.22e-05&   4.07  & 1.97e-06&   5.64 &  2.87e-04 &  2.01\\
        &4.42e-02 &  2.59e-04&   3.07 &  3.14e-06 &  4.06  & 4.41e-08 &  5.48 &  7.14e-05 &  2.01\\
        &2.21e-02&   3.09e-05 &  3.07  & 1.92e-07 &  4.04  & 1.11e-09&   5.32 &  1.79e-05  & 1.99\\

\hline

         &3.54e-01  & 3.46e-01  &     -- &  4.97e-02 &   --  & 9.70e-03  &     --  & 4.43e-03 &   --     \\
         & 1.77e-01 &  5.86e-02 &  2.56&  3.23e-03 &  3.94 &  1.45e-04 &  6.06 &  1.11e-03 &  2.00\\
  $p=4$ & 8.84e-02  & 3.37e-03 &  4.12 &  9.82e-05 &  5.04 &  1.91e-06 &  6.25 &  2.83e-04 &  1.97\\
        & 4.42e-02 &  2.03e-04&   4.06 &  3.06e-06 &  5.00 &  2.85e-08 &  6.07 &  7.15e-05 &  1.99\\
        &2.21e-02 &  1.25e-05 &  4.01&  9.61e-08 &  4.99&  4.41e-10  & 6.01  & 1.79e-05  & 2.00\\
\hline

       &3.54e-01 &  5.46e-01 &      -- &  9.69e-02 &        --  & 1.38e-02 &       -- &  4.83e-03 &   --     \\
 $p=5$ &1.77e-01 &  1.68e-01 &  1.70 &  1.03e-02&   3.23 &  3.67e-04 &  5.24&   1.13e-03 &  2.10\\
       &8.84e-02 &  6.03e-03 &  4.80 &  1.97e-04&   5.71&   3.70e-06 &  6.63&   2.84e-04 &  1.99\\
       &4.42e-02 &  1.89e-04 &  5.00 &  3.10e-06 &  5.99 &  2.91e-08 &  6.99 &  7.08e-05 &  2.00\\
 \hline
 \end{tabular}
 \end{threeparttable}
\end{table}

\begin{table}[h]%[htbp]
\caption{Numerical error and rate of convergence for the $L^p$-PDWG method with $k=1, s=k-1$ for Example 2.}\label{ex3:2} \centering
\tiny
\begin{threeparttable}
        \begin{tabular}{c | c c c c c c c }
        \hline
    & h&   $\|\lambda_0\|_{1,p}$  & rate& $\|\lambda_0\|_{0,p}$ & rate& $\|e_h\|_{0,q}$ &rate   \\
       \hline \cline{2-8}

        &3.54e-01 &  7.72e-02 &   -- &  1.04e-02 &      --  & 2.10e-01 &    --     \\
        & 1.77e-01&   8.44e-02&  -0.13 &  9.49e-03&   0.14 &  1.22e-01 &  0.78\\
  $p=1$ & 8.84e-02&   8.90e-02&  -0.08 &  9.27e-03&   0.03 &  7.14e-02&   0.78\\
        &4.42e-02 &  9.16e-02 & -0.04&  9.31e-03 & -0.00 &  4.49e-02 &  0.67\\
        &2.21e-02 &  9.29e-02&  -0.02&   9.42e-03 & -0.01 &  2.80e-02 &  0.68\\

\hline

        &3.54e-01 &  1.60e-01 &   --  &  3.09e-02 &  -- &  4.96e-02 &           \\
        &1.77e-01 &  4.76e-02 &  1.75 &  7.83e-03 &  1.98 &  2.47e-02 &  1.00\\
  $p=2$ & 8.84e-02&   1.31e-02&   1.86&   2.02e-03&  1.96 &  1.19e-02&   1.06\\
        &4.42e-02 &  3.47e-03 &  1.92 &  5.14e-04 &  1.97&  5.64e-03 &  1.07\\
        &2.21e-02&   8.90e-04 &  1.96&  1.30e-04 &  1.99 &  2.74e-03&   1.04\\
\hline

        &3.54e-01 &  3.16e-01 &  -- &  7.09e-02 &    -- &  4.43e-02 &          \\
        &1.77e-01 &  4.17e-02 &  2.92 &  8.68e-03 &  3.03 &  2.07e-02 &  1.09\\
  $p=3$ & 8.84e-02&   3.45e-03 &  3.60 &  7.93e-04&  3.45&   9.88e-03&   1.07\\
        &4.42e-02 &  3.04e-04  & 3.51 & 7.64e-05  & 3.38 &  4.86e-03  & 1.02\\
        &2.21e-02&   3.12e-05 &  3.28 &  8.05e-06  & 3.25  & 2.42e-03 &  1.01\\
\hline

        &3.54e-01  & 9.35e-02 &    --  &   4.03e-02 &      -- & 5.01e-02 &   --  \\
         &1.77e-01 &  4.59e-02 &  1.03 &  1.02e-02 &  1.99  & 2.10e-02&   1.25\\
  $p=4$ & 8.84e-02&   3.84e-03 &  3.58 &  9.62e-04 &  3.40  & 9.70e-03&   1.12\\
        &4.42e-02 &  2.39e-04 &  4.01 &  5.76e-05 &  4.06  & 4.73e-03 &  1.03\\
        &2.21e-02 &  1.49e-05 &  4.00&  3.55e-06 &  4.02 &  2.35e-03 &  1.01\\

\hline
          & 3.54e-01 &  6.61e-03&  --&   8.87e-03 &     -- &  2.51e-01 &  --  \\
          &1.77e-01 &  2.82e-02 & -2.09&  4.51e-03 &  0.98 &  2.14e-02  & 3.55\\
   $p=5$ &8.84e-02  & 7.21e-03 &  1.97 & 1.76e-03  & 1.36  & 9.70e-03  & 1.14\\
          & 4.42e-02&   2.51e-04&   4.85&   5.98e-05&   4.88 &  4.70e-03 &  1.04\\
          &2.21e-02 &  7.85e-06&   5.00 &  1.87e-06 &  5.00 &  2.34e-03 &  1.01\\

\hline
 \end{tabular}
 \end{threeparttable}
\end{table}

\begin{example}\label{Example3}
Let $\Omega=(0,1)^2$ and $\Omega_1=(0.25,0.75)^2,\Omega_2=\Omega\backslash\bar\Omega_1$.
Consider the following problem: Find an unknown function $u$ satisfying
\begin{equation}\label{model2}
\begin{split}
-\nabla\cdot(\alpha\nabla u)+\bbeta u =&f,   \quad \ {\rm in}\  \Omega_1\cup\Omega_2,\\
 u=&g,  \quad\ {\rm on}\ \partial\Omega\ \\
[\![  u]\!]_{\Gamma}=0,\ \ \ [\![(\alpha \nabla u-\bbeta u)\cdot \bn]\!]_{\Gamma} =&\psi, \quad\   {\rm on}\ \Gamma={\partial\Omega_1}\cap{\partial\Omega_2},
\end{split}
\end{equation}
where $\bn$ is the unit outward normal of $\Gamma$ with respect to $\Omega_1$, and $[\![  u]\!]_{\Gamma}$ denotes the jump of $u$ across the interface $\Gamma$.
 We take $f=0,\bbeta=(0,0)$ and a piecewise constant function $\alpha$ defined as  $\alpha|_{\Omega_1}=5, \ \alpha|_{\Omega_2}=1$.
 The functions $g$ and $\psi$ are chosen  so that the exact solution to this problem is
given by  $u= e^x\cos(y)+10$.
\end{example}

Based on the variational formulation, we numerically solve the above problem by slightly modifying our algorithm \eqref{32}-\eqref{2} as follows: Find $(u_h;\lambda_h)\in M_h \times W_{h}^0$, such that
\begin{eqnarray*}
s(\lambda_h, \sigma)+b(u_h, \sigma)&=& (f, \sigma_0)- \langle g, \sigma_n \rangle_{\partial \Omega}+\langle \psi, \sigma_b \rangle_{\Gamma}, \qquad \forall \sigma\in W_{h}^0,\\
b(v, \lambda_h)&=&0,\qquad \qquad \qquad   \qquad \qquad \forall v\in M_h,
\end{eqnarray*}
  where
\begin{equation*}
\begin{split}
s(\lambda, \sigma)=&\sum_{T\in{\cal T}_h}h_T^{1-2p}\int_{\partial T} |\lambda_0-\lambda_b|^{p-1}sgn(\lambda_0-\lambda_b)(\sigma_0-\sigma_b)ds\\
&+ \sum_{T\in{\cal T}_h} h_T^{1-p}\int_{\partial T}  |\alpha\nabla \lambda_0 \cdot \bn-\lambda_n|^{p-1}sgn (\alpha\nabla \lambda_0 \cdot \bn-\lambda_n) (\alpha\nabla \sigma_0 \cdot \bn-\sigma_n)ds.\\
%& + \int_{T} |\lambda_0|^{p-1}sgn(\lambda_0) \ \sigma_0 dT + \int_{T} |\nabla\lambda_0|^{p-1}sgn(\nabla\lambda_0) \nabla\sigma_0 dT,\\
b(u, \lambda)=&\sum_{T\in{\cal T}_h}(u, - \bbeta\cdot\nabla_w\lambda - \alpha\Delta_w \lambda)_T.
\end{split}
\end{equation*}

Tables  \ref{4}-\ref{5} show the numerical performance for the primal variable $u_h$ and the dual variable $\lambda_0$ for $k=1,2$ and $p=1,\ldots,5$. It can be seen that, for both the linear (i.e., $k=1$) and quadratic (i.e., $k=2$) PDWG methods, the convergence rate for the error $\|e_h\|_{0,q}$ is of order $\mathcal{O}(h^{s+1})$. This numerical experiment suggests that the theoretical estimate in Theorem \ref{theoestimate} should hold true for elliptic problems with piecewise constant diffusions. Analogously to Example \ref{Example1}, the error $\|\lambda_0\|_{m,p}, m=0,1,2,$ for the dual variable converges to zero at the rate of ${\mathcal O}(h^{p+2-m})$ when $k=2, s=k-1$ and $p\ge 1$. As shown in Tables \ref{6}-\ref{5}, the rate of convergence varies in ways that depend on the value of $p$ for $k=1, s=k-1$ and $k=2,s=k-2$. Note that no convergence was seen from Tables \ref{6} -\ref{5} for the dual variable $\lambda_0$ for the case of $p=1,s=0$.

 \begin{table}[h]%[htbp]
\caption{Numerical error and rate of convergence for the $L^p$-PDWG method with $k=2, s=k-1$ for Example 3.}\label{4} \centering
\tiny
\begin{threeparttable}
        \begin{tabular}{c |c c c c c c c cc }
        \hline
   & h&   $\|\lambda_0\|_{2,p}$ & rate& $\|\lambda_0\|_{1,p}$  & rate& $\|\lambda_0\|_{0,p}$ & rate& $\|e_h\|_{0,q}$ &rate   \\
       \hline \cline{2-10}

           &  3.54e-01&   1.19e-02&        --  & 9.22e-04    &       -- & 8.57e-05 &      --  & 7.176e-02&    --         \\
            & 1.77e-01 &  4.90e-03 &  1.29 &  1.66e-04&   2.47&   1.11e-05  & 2.96&   1.85e-02 &  1.96 \\
$p=1$  & 8.84e-02 &  2.25e-03 &  1.12 &  3.34e-05 &  2.31&   1.40e-06 &  2.99 &  4.58e-03 & 2.01 \\
           & 4.42e-02 &  1.09e-03 &  1.04 &  7.57e-06 &  2.14&   1.75e-07  & 2.99 &  1.18e-03 &  1.95 \\
           &  2.21e-02 &  5.41e-04  & 1.01&   1.82e-06  & 2.06   &2.20e-08 &  3.00&   3.16e-04  & 1.91 \\

\hline
             &   3.54e-01&   5.02e-02&      -- & 7.14e-03  &        --   &7.86e-04  &    --  &1.21e-02&           -- \\
             &  1.77e-01 &  9.29e-03&   2.43&   6.07e-04 &  3.55&   4.77e-05 &  4.04&   3.06e-03 &  1.98 \\
$p=2$  &  8.84e-02&   1.97e-03 &  2.24&   5.62e-05  & 3.43&   2.91e-06&   4.03&   7.68e-04 &  2.00 \\
          &   4.42e-02 &  4.57e-04&   2.11&   5.90e-06  & 3.25 &  1.79e-07&   4.02&   1.92e-04   &2.00 \\
          &  2.21e-02 &  1.11e-04&   2.04&   6.79e-07 &  3.12&  1.11e-08 &  4.01&  4.80e-05   &2.00 \\
\hline

           &3.54e-01 &  6.98e-02&          --   &1.30e-02  &      --   &1.37e-03       &   --     & 7.30e-03&             \\
             & 1.77e-01&   4.48e-03 &  3.96  & 4.41e-04  & 4.88&   2.85e-05  & 5.59 &  1.66e-03  & 2.13 \\
 $p=3$  & 8.84e-02 &  3.59e-04 &  3.64& 1.78e-05 &  4.63&  5.95e-07  & 5.58&   3.70e-04  & 2.17 \\
            & 4.42e-02 &  3.91e-05 &  3.20 &  8.42e-07 &  4.40 &  1.29e-08  & 5.53&   7.99e-05 &  2.21 \\
            & 2.21e-02 &  4.94e-06 &  2.98&   4.63e-08 &  4.19&  2.98e-10 &  5.43 &  1.75e-05 &  2.19 \\
\hline
              &3.54e-01&   9.83e-03 &         -- &1.97e-03  &     --  &1.94e-04 &  -- & 5.61e-03 &     --       \\
              &1.77e-01 &  3.93e-04  & 4.64&   3.17e-05&   5.95&  1.54e-06&   6.98 &  1.11e-03 &  2.33 \\
  $p=4$ &  8.84e-02 &  2.34e-05 & 4.07&   8.58e-07&   5.21&   1.69e-08 &  6.50 & 2.71e-04 &  2.04 \\
             &4.42e-02 &  1.43e-06 &  4.03&   2.81e-08  & 4.93&   2.41e-10&   6.14&   6.76e-05 &  2.00 \\
             & 2.21e-02 &  8.86e-08 &  4.02  & 8.98e-10 &  4.97&   3.67e-12 &  6.03&   1.67e-05&   2.02\\
      \hline
             %   & 7.07e-01   & 4.17e-00   &   --    &1.32e-00   &   --    &    1.15e-01  &  --     & 4.92e-01  &  --  \\
                &3.54e-01  &  6.28e-03   & --  &1.17e-03   &-- &  8.38e-05    & --  &  2.16e-02 &  -- \\
  $p=5$   &1.77e-01   & 1.70e-04   & 5.21  & 1.25e-05  &  6.54  &   4.63e-07 &   7.50  &  4.88e-03 & 2.15 \\
              & 8.84e-02   & 4.92e-06  & 5.11   & 2.05e-07  &  5.94 &   3.49e-09   & 7.05   & 1.25e-03  & 1.96 \\
              &4.42e-02  &1.49e-07   & 5.04    &  3.23e-09 &5.98  & 2.71e-11       &   7.01  & 3.19e-04 &   1.98 \\
  \hline
 \end{tabular}
 \end{threeparttable}
\end{table}

  \begin{table}[h]%[htbp]
\caption{Numerical error and rate of convergence for the $L^p$-PDWG method with $k=2, s=k-2$ for Example 3.}\label{6} \centering
\tiny
\begin{threeparttable}
        \begin{tabular}{c |c c c c c c c cc }
        \hline
   & h&   $\|\lambda_0\|_{2,p}$ & rate& $\|\lambda_0\|_{1,p}$  & rate& $\|\lambda_0\|_{0,p}$ & rate& $\|e_h\|_{0,q}$ &rate   \\
       \hline \cline{2-10}
            & 3.54e-01 &  8.36e-00  &      --& 3.79e-00 &       -- &  3.85e-01   &       --&   6.45e-01 &     --\\
$p=1$  & 1.77e-01 &  1.20e-00 &  2.79&   9.03e-01 &  2.07&   8.98e-02&   2.10&  3.62e-01&   0.83\\
          &8.84e-02  & 1.09e-00 &  0.15 &  8.55e-01 &  0.08  & 8.19e-02 &  0.13  & 3.12e-01&   0.21\\
         & 4.42e-02 &  1.09e-00 &  0.00 &  8.76e-01 & -0.03&   8.30e-02 & -0.02  & 2.69e-01&   0.22\\
\hline
        &  3.54e-01  &  6.49e-01&        --  &  5.77e-01    &     --      & 8.48e-02     &   --  & 1.39e-01    &   --\\
              & 1.77e-01  &  2.19e-01&    1.57  &   1.72e-01    & 1.75   & 2.34e-02   & 1.86 &    8.11e-02 &   0.77\\
 $p=2$  &   8.84e-02  &  1.01e-01&    1.11 &   5.10e-02   & 1.76   & 6.66e-03   & 1.82  &  4.59e-02 &   0.82\\
               &   4.42e-02 &   4.56e-02&    1.14&   1.46e-02  &  1.81  &  1.86e-03  &  1.84&   2.41e-02 &   0.93\\
                &   2.21e-02 &   1.98e-02&    1.20&    3.96e-03  &  1.88  &  4.97e-04  &  1.91&   1.21e-02  &  0.99\\
\hline
           &  3.54e-01  &  1.03e+00  &      --  &  8.11e-01  &          --   & 1.21e-01&     --&    1.21e-01&             --\\
             &  1.77e-01  &  4.74e-01 &   1.12 &   1.64e-01&    2.30&    2.35e-02 &   2.36  &  6.25e-02  &  0.96\\
 $p=3$  & 8.84e-02 &   1.30e-01 &   1.87&   1.37e-02 &   3.58&    1.85e-03&    3.67&   2.93e-02  &  1.10\\
              &   4.42e-02 &   2.63e-02 &   2.30 &   9.76e-04 &   3.82&    1.22e-04&    3.92  &  1.43e-02 &   1.04\\
             &  2.21e-02 &   5.11e-03&    2.37 &   6.87e-05 &   3.83&    7.83e-06 &   3.97&   7.07e-03&    1.01\\
\hline
           & 3.54e-01&    2.90e-02   &    --      &    7.04e-02 &      --  &  7.84e-03   &     --   & 1.19e-01 &            --\\
           & 1.77e-01&    9.24e-01 &  -5.00 &   1.58e-01&   -1.16&    2.34e-02  & -1.58&    5.82e-02 &   1.03\\
$p=4$  &  8.84e-02&    7.81e-02 &   3.60 &   4.66e-03 &   5.08 &   5.39e-04  &  5.44 &   2.74e-02  &  1.09\\
            &   4.42e-02&    3.95e-03 &   4.31&    1.11e-04&    5.39 &   1.32e-05 &   5.35&   1.37e-02  &  1.00\\
   & 2.21e-02&    2.83e-04&    3.80&    3.59e-06&    4.96&   5.06e-07 &   4.71&   6.84e-03  &  1.00\\
  \hline

            &3.54e-01&   1.66e-00&     --  & 2.71e-01   &    -- & 3.93e-02 &       --    &  1.17e-01 &   --  \\
$p=5$   &1.77e-01&   5.55e-02  & 4.90&  9.24e-03 &  4.88&  1.15e-03 &  5.09&   5.40e-02  & 1.20\\
           & 8.84e-02 &  1.83e-03&   4.93 &  8.84e-05 &  6.71&   8.65e-06 &  7.06&  2.69e-02  & 1.00\\
          & 4.42e-02 &  8.89e-05&   4.36 &  1.86e-06 &  5.58&   1.50e-07 &  5.85&   1.35e-02   &1.00\\
    \hline
 \end{tabular}
 \end{threeparttable}
\end{table}

 \begin{table}[h]%[htbp]
\caption{Numerical error and rate of convergence for the $L^p$-PDWG method with $k=1, s=k-1$ for Example 3.}\label{5}\centering
\tiny
\begin{threeparttable}
        \begin{tabular}{c | c c c c c c c }
        \hline
    & h&   $\|\lambda_0\|_{1,p}$  & rate& $\|\lambda_0\|_{0,p}$ & rate& $\|e_h\|_{0,q}$ &rate   \\
       \hline \cline{2-8}

              &3.54e-01 &  4.35e-01&     -- & 4.89e-02    &        --   &5.31e-01  &    --      \\
              &1.77e-01 &  2.52e-01 &  0.79&   2.66e-02&   0.88 &  4.04e-01 &  0.39\\
 $p=1$  &8.84e-02 &  2.46e-01 &  0.03 &  2.43e-02&   0.13 &  3.41e-01  & 0.25\\
             & 4.42e-02&   2.57e-01 & -0.06&   2.45e-02&  -0.01&   2.95e-01 &  0.21\\
            & 2.21e-02 &  2.71e-01&  -0.07&   2.55e-02 & -0.05&   2.56e-01 &  0.20\\

  \hline

    &3.54e-01&   4.70e-01 &       -- &  7.58e-02&       -- &  1.47e-01 &  -- \\
              &1.77e-01 &  1.48e-01&   1.67&   2.13e-02&   1.83 &  8.87e-02 &  0.73\\
  $p=2$ & 8.84e-02&  4.52e-02 &  1.71&   6.04e-03 &  1.82 &  5.41e-02 &  0.71\\
             &4.42e-02&   1.33e-02&   1.77&  1.71e-03 &  1.82 & 3.11e-02  & 0.80\\
            & 2.21e-02&   3.74e-03&   1.83  & 4.72e-04&   1.86&   1.67e-02&   0.90\\
\hline
              &  3.54e-01 &  6.49e-01 &   --     &1.08e-01  &      --   &  1.26e-01  &   --\\
              & 1.77e-01  & 1.40e-01&   2.22 &  2.08e-02&   2.38   & 6.55e-02 &  0.94\\
  $p=3$ & 8.84e-02&   1.22e-02&   3.52&   1.72e-03&   3.60 &  3.10e-02&   1.08\\
             & 4.42e-02&   8.73e-04&   3.80 &  1.20e-04&   3.84  &  1.48e-02 &  1.07\\
            &2.21e-02&   5.78e-05  & 3.92 &  7.69e-06 &  3.96  & 7.17e-03  & 1.04\\
      \hline

           &3.54-01&   8.70e-02     &    --   &1.27e-02 &          --  &1.20e-01  &       --   \\
           & 1.77e-01 &  1.30e-01  &-0.58&   2.08e-02&  -0.71  & 6.11e-02  & 0.98\\
$p=4$ & 8.84e-02&   3.35e-03 &  5.28&  4.82e-04&   5.43 &  2.75e-02 &  1.15 \\
           & 4.42e-02 &  6.33e-05 &  5.73&  9.16e-06&   5.72&   1.37e-02  & 1.01\\
          & 2.21e-02 &  1.87e-06 &  5.08 &  3.38e-07 &  4.76 & 6.83e-03 &  1.00\\

 \hline

             &3.54e-01  &  2.83e-01  &    --  & 4.90e-02   &        --   &1.19e-01    &        -- \\
             &1.77e-01   & 6.13e-03   & 5.53  & 9.44e-04  &  5.70  &  5.40e-02  & 1.15\\
 $p=5$ &  8.84e-02  &  3.75e-05  &  7.35 &   6.05e-06  &  7.29 &   2.69e-02 &   1.00\\
            & 4.42e-02  &  7.05e-07   & 5.73  &  1.34e-07   & 5.49  &  1.35e-02  &  1.00\\
            &2.21e-02   & 2.28e-08  &  4.95  & 4.35e-09   & 4.95 &  6.73e-03  &  1.00\\
 \hline
 \end{tabular}
 \end{threeparttable}
\end{table}

\begin{example}
In this test, we consider the model problem \eqref{model2} in the domain
 $\Omega=(0,1)^2=\Omega_1\cup\Omega_2$, where $\Omega_1$ is the circular domain centered at $(0.5,0.5)$ with radius $0.25$ and $\Omega_2=\Omega\backslash\Omega_1$. The diffusive and convective terms are given by
\[
   \alpha|_{\Omega}=1, \ \ \bbeta|_{\Omega_1}=(0,1),\ \ \bbeta|_{\Omega_2}=(1,0).
\]
 The functions $f, g$ and $\psi$ are chosen so that the exact solution to this problem is $u= e^{x^2+y}$.
\end{example}
   
The same algorithm as that for Example \ref{Example3} was employed for solving this problem. Tables  \ref{ex4:2}-\ref{ex4:1} show the performance of the $L^p$-PDWG for $u_h$ and $\lambda_h$ for $k=1,2$ and $p=1,\ldots,4$. It can be seen that the convergence for the error $\|e_h\|_{0,q}$ is of order $\mathcal{O}(h^{s+1})$ for $k=1,2$ and $p>1$, which suggests that the results of Theorem \ref{theoestimate} may hold true for elliptic problems with piecewise constant or smooth convection. For $\|\lambda_0\|_{m,p}$, Tables \ref{ex4:2}-\ref{ex4:1} demonstrate a convergence of $\mathcal{O}(h^{p+2-m})$ when $k=2$ and $\mathcal{O}(h^{p})$ when $k=1$. It was observed the numerical convergence for the dual variable $\lambda_0$ performs significantly better than the theory shown in Theorems \ref{theo:1}-\ref{theo:2}. Tables \ref{ex4:2}-\ref{ex4:1} also show the convergence of the numerical approximations for $p=1$ and $k=1,2$. The convergence for the dual variable $\lambda_0$ varies according to the choice of $M_h$. For $s=1$,
the convergence for $\|\lambda_0\|_{m,p}$ is at the rate of $\mathcal{O}(h^{p+2-\max(1,m)})$, but for $s=0$, no convergence was observed.

 \begin{table}[h]%[htbp]
\caption{Numerical error and rate of convergence for the $L^p$-PDWG method with $k=2, s=k-1$ for Example 4.}\label{ex4:2} \centering
\tiny
\begin{threeparttable}
        \begin{tabular}{c |c c c c c c c cc }
        \hline
   & h&   $\|\lambda_0\|_{2,p}$ & rate& $\|\lambda_0\|_{1,p}$  & rate& $\|\lambda_0\|_{0,p}$ & rate& $\|e_h\|_{0,q}$ &rate   \\
       \hline \cline{2-10}

      &  3.54e-01  &  4.52e-02 &   --&     5.87e-03&    -- &   3.69e-03  &  --  &  2.34e-01&    --        \\
      &  1.77e-01  &  1.62e-02 &   1.48&   1.27e-03 & 2.21&   8.91e-04 &  2.05 &  8.70e-02 &  1.43\\
 $p=1$&  8.84e-02  &  6.81e-03 &   1.25&   2.94e-04 & 2.11&   2.21e-04 &  2.01 &  2.73e-02&  1.67\\
      & 4.42e-02  &  3.13e-03 &   1.12&  7.05e-05 &  2.06&  5.52e-05&  2.00&   7.88e-03 & 1.79\\
      & 2.21e-02  &  1.51e-03 &   1.06&    1.72e-05&   2.03&    1.38e-05&  2.00&   2.13e-03&   1.89\\

\hline

       &3.54e-01 &  2.10e-01 &    -- &  3.29e-02&   -- &  1.24e-02 &    --&  3.44e-02 &          \\
       &1.77e-01 &  3.87e-02 &  2.44 &  2.86e-03&   3.52&   7.94e-04 &  3.97&   8.92e-03&   1.95\\
 $p=2$ & 8.84e-02&   7.90e-03&   2.29 &  2.62e-04&   3.45&   4.99e-05&   3.99&   2.25e-03&   1.99\\
       &4.42-02 &  1.75e-03 &  2.17 &  2.62e-05 &  3.32&   3.12e-06 &  4.00&   5.636e-04 &  2.00\\
       &2.21e-02 &  4.14e-04 &  2.08 &  2.88e-06&   3.19&   1.95e-07&   4.00 &  1.41e-04 &  2.00\\

\hline

         &3.54e-01&   4.75e-01&    --&    9.00e-02 &    -- &   2.47e-02 &    -- &  2.38e-02&        \\
         &1.77e-01&   7.43e-02&   2.68&   7.34e-03 &  3.62&   7.95e-04&   4.96&   5.18e-03& 2.20\\
  $p=3$ & 8.84e-02&    6.28e-03&  3.56 &  3.22e-04 &  4.51&   1.54e-05 &  5.69&   1.19e-03&  2.12\\
        &4.42e-02&  8.25e-04&   2.93&     1.64e-05&  4.30 &    3.08e-07& 5.65&   2.76e-04 & 2.10\\
          &2.21e-02&   4.48e-05&   2.75&     9.48e-07&  4.11&  6.75e-09 & 5.51&   6.55e-05 &  2.08\\

\hline

       &3.54e-01&   2.87e-01&        -- &  6.20e-02 &       -- &  1.34e-02&        -- & 1.97e-02&       \\
       &1.77e-01 &  1.21e-01 &  1.25 &  1.37e-02 &  2.18 & 8.25e-04 &  4.02 & 4.53e-03&   2.12\\
$p=4$  & 8.84e-02&   1.91e-02&   2.66 &  5.84e-04&   4.55&  1.18e-05&   6.13&   9.90e-04&   2.19\\
       &4.42e-02&   1.39e-03 &     3.78&     1.97e-05&   4.89&  1.64e-07&  6.17&   2.41e-04&   2.04\\

\hline
 \end{tabular}
 \end{threeparttable}
\end{table}

\begin{table}[h]%[htbp]
\caption{Numerical error and rate of convergence for the $L^p$-PDWG method with $k=1, s=k-1$ for Example 4.}\label{ex4:1}\centering
\tiny
\begin{threeparttable}
        \begin{tabular}{c | c c c c c c c }
        \hline
    & h&   $\|\lambda_0\|_{1,p}$  & rate& $\|\lambda_0\|_{0,p}$ & rate& $\|e_h\|_{0,q}$ &rate   \\
       \hline \cline{2-8}
   %    & 3.5355e-01 &  3.0296e+12 &     0 &        6.1865e+11    &        0  & 2.3374e+00 &       \\
       & 1.77e-01 &  1.07e-00 &  --   &   8.89e-02&   -- &  1.09e-00 &  --\\
 $p=1$ & 8.84e-02 &  8.14e-01 &  0.39 &  4.59e-02 &  0.95 &  7.15e-01&   0.60\\
       & 4.42e-02 &  8.01e-01&   0.02 &  4.37e-02 &  0.07 &  5.42e-01&   0.40\\
       & 2.21e-02 &  8.21e-01 & -0.03 &  4.50e-02 & -0.04&   4.52e-01&   0.26\\
\hline

        & 3.54e-01&   1.16e-00 &     -- &  1.33e-01 &    -- &  3.01e-01&            \\
        &1.77e-01 &  3.46e-01 &  1.74 &  3.50e-02  & 1.93 &  1.57e-01 &  0.94\\
  $p=2$ & 8.84e-02&   1.09e-01 &  1.67  & 1.07e-02&   1.80 &  7.68e-02&   1.03\\
        & 4.42e-02 &  3.23e-02 &  1.75 &  2.78e-03 &  1.85&   3.54e-02 &  1.12\\
        &2.21e-02 &  8.87e-03  & 1.86 &  7.31e-04&   1.93 &  1.61e-02 &  1.13\\

\hline

        &3.54e-01 &  9.50e-01 &     --&   1.14e-01 &    -- &  2.29e-01 &  --         \\
        &1.77e-01 &  2.37e-01 &  2.00 &  2.91e-02 &  1.96&   1.16e-01 &  0.98\\
  $p=3$ &8.84e-02 &  2.69e-02 &  3.14 &  3.59e-03 &  3.02 &  5.40e-02 &  1.10\\
        & 4.42e-02 &  2.08e-03 &  3.69 &  3.33e-04 &  3.43 &  2.60e-02 &  1.05\\
        & 2.21e-02 &  1.61e-04 &  3.69 &  3.17e-05 &  3.40 &  1.29e-02 &  1.02\\

\hline

       & 3.54e-01 &  7.06e-00 &   --  & 1.24e-00 &    --  & 5.98e-01 &      --     \\
       & 1.77e-01 &  1.67e-01 &  5.41  & 1.77e-02 &  6.13 &  1.35e-01 &  2.14\\
 $p=4$ & 8.84e-02&   1.12e-02 &  3.90 &  2.44e-03 &  2.86 &  5.41e-02 &  1.32\\
       & 4.42e-02&   1.03e-03 &  3.44 &  1.99e-04 &  3.61 &  2.61e-02  & 1.05\\
       & 2.21e-02 &  6.60e-05 &  3.96 &  1.21e-05 &  4.04 &  1.27e-02  & 1.04\\

 \hline
 \end{tabular}
 \end{threeparttable}
\end{table}

\vfill\eject

 %%%%%%%%%%%%%%%%%%%%%%%%%%%%%%%%%%%%%%%%%%%%%%%%%%%%%%%%%%%


\begin{thebibliography}{99}
\bibitem{b1970} {\sc I. Babuska}, {\em The finite element method for elliptic equations with discontinuous coefficients}, Computing 5,  pp. 207-213, 1970.

\bibitem{b1974} {\sc F. Brezzi}, {\em On the existence, uniqueness, and approximation of  saddle point problems arising from Lagrange multipliers}, RAIRO, 8, pp. 129-151, 1974.

\bibitem{r1988} {\sc R. Duran}, {\em Error analysis in $L^p$, $1\leq p\leq \infty$, for mixed finite element methods for linear and quasi-linear elliptic problems}, Mathematical Modeling and Numerical Analysis, vol. 22, pp.  371-387, 1988.

\bibitem{Vogel}
 {\sc  C.  Vogel and M. Oman
}, {\em Iterative Methods For Total Variation Denoising},
 SIAM Journal on Scientific Computing, vol. 17, pp. 227-238, 1996.

\bibitem{w2020} {\sc  C. Wang}, {\em A new primal-dual weak Galerkin finite element method for ill-posed elliptic Cauchy problems}, Journal of Computational and Applied Mathematics, vol.  371, pp. 112629, 2020.

\bibitem{w1} {\sc C. Wang}, {\em New discretization schemes for time-harmonic Maxwell equations by weak Galerkin finite element methods}, Journal of Computational and Applied Mathematics, vol. 341, pp. 127-143, 2018.

\bibitem{w3} {\sc C. Wang and J. Wang}, {\em  Discretization of div-curl systems by weak Galerkin finite element methods on polyhedral partitions}, Journal of Scientific Computing, vol. 68, pp. 1144-1171, 2016.

\bibitem{WW-mathcomp}
 {\sc C. Wang and J. Wang}, {\em A primal-dual weak Galerkin finite element method for second order elliptic equations in non-divergence form}, Math. Comp.,  vol. 87, pp. 515-545, 2018.

\bibitem{w5} {\sc C. Wang and J. Wang}, {\em  A hybridized formulation for weak Galerkin finite element methods for biharmonic equation on polygonal or polyhedral meshes}, International Journal of Numerical Analysis and Modeling, vol. 12, pp. 302-317, 2015.

\bibitem{w6} {\sc J. Wang and C. Wang}, {\em  Weak Galerkin finite element methods for elliptic PDEs}, Science China, vol. 45, pp. 1061-1092, 2015.

 \bibitem{w7}  {\sc C. Wang and J. Wang}, {\em  An efficient numerical scheme for the biharmonic equation by weak Galerkin finite element methods on polygonal or polyhedral meshes}, Journal of Computers and Mathematics with Applications, vol. 68, 12, pp. 2314-2330, 2014.

 \bibitem{ww2017} {\sc C. Wang, and J. Wang}, {\em A primal-dual weak Galerkin finite element method for Fokker-Planck type equations},  SIAM Journal of Numerical Analysis, vol. 58(5), pp. 2632-2661, 2020.

\bibitem{ww2018} {\sc C. Wang and J. Wang}, {\em Primal-dual weak Galerkin finite element methods for elliptic Cauchy problems}, Computers and Mathematics with Applications, vol. 78, pp. 905-928, 2019.

\bibitem{wwhyperbolic} {\sc C. Wang, and J. Wang}, {\em A primal-dual finite element method for first-order transport problems},  Journal of Computational Physics, vol. 417, 109571, 2020.


\bibitem{w4} {\sc C. Wang, J. Wang, R. Wang and R. Zhang}, {\em  A Locking-Free Weak Galerkin Finite Element Method for Elasticity Problems in the Primal Formulation}, Journal of Computational and Applied Mathematics, Vol. 307, pp. 346-366, 2016.

 \bibitem{w2} {\sc C. Wang and H. Zhou}, {\em  A weak Galerkin finite element method for a type of fourth order problem arising from fluorescence tomography}, Journal of Scientific Computing, vol. 71(3), pp. 897-918, 2017.
%
%  \bibitem{wy3655}
%{\sc J. Wang and X. Ye}, {\em A weak Galerkin mixed finite element
%method for second-order elliptic problems},  Math. Comp., 83 (2014), pp. 2101-2126.

\bibitem{Wang:YeWG2013} {\sc J. Wang and X. Ye},
{\em A weak Galerkin finite element method for second-order elliptic problems},   J. Comput. Appl. Math., vol. 241, pp. 103-115, 2013.

\bibitem{wang-ye-2014} {\sc J. Wang and X. Ye},
 {\em A weak Galerkin mixed finite element method for second-order elliptic problems}, Math. Comp., vol. 83, pp. 2101-2126, 2014.

\bibitem{wang-ye-2015} {\sc J. Wang and X. Ye},  {\em A weak Galerkin finite element method for the Stokes equations}, Advances in Computational Mathematics, vol. 42, pp. 155-174, 2016.

 \end{thebibliography}
\end{document}